\documentclass{article}
\usepackage[T1]{fontenc}
\usepackage[utf8]{inputenc}

\usepackage{amsmath,amssymb,amsthm,mathtools,bm}
\usepackage{thmtools,thm-restate}

\usepackage{graphicx}
\usepackage[dvipsnames]{xcolor}
\usepackage{tikz}
\usetikzlibrary{arrows.meta,calc,decorations.pathmorphing,positioning}

\usepackage{float}
\usepackage{booktabs}
\usepackage{array}
\usepackage{enumitem}

\usepackage[hidelinks]{hyperref}

\makeatletter
\newcommand{\rep@title}{}
\newtheorem*{rep@theorem}{\rep@title}
\newcommand{\newreptheorem}[2]{%
\newenvironment{rep#1}[1]{%
 \def\rep@title{#2 \ref{##1}}%
 \begin{rep@theorem}}%
 {\end{rep@theorem}}}
\makeatother

\newcolumntype{C}[1]{>{\centering\let\newline\\\arraybackslash\hspace{0pt}}m{#1}}

\usepackage{subcaption}
\usepackage[dvipsnames]{xcolor} 
\usepackage{tikz} 
\usetikzlibrary{calc}

\usepackage{mathdots} 
\usepackage[normalem]{ulem}
\usepackage{xfrac}
\usepackage{multirow} 

\hypersetup{
  colorlinks = true, 
  urlcolor = blue, 
  linkcolor = blue, 
  citecolor = red 
}

\definecolor{Dgreen}{RGB}{2,100,64}

\newcommand{\Z}{\mathbb{Z}}

\newcommand{\x}{\mathbf{x}}

\newtheorem{definition}{Definition}
\newtheorem{lemma}{Lemma}
\newtheorem{theorem}{Theorem}
\newtheorem{cor}{Corollary}
\newtheorem{conjecture}{Conjecture}
\newtheorem{prop}{Proposition}
\theoremstyle{remark}
\newtheorem{example}{Example}
\newtheorem{remark}{Remark}
\newtheorem*{remark*}{Remark}

\usepackage{biblatex}
\title{Unimodality and Cluster Algebras from Surfaces}
\author{Wonwoo Kang, Kyeongjun Lee, Eunsung Lim}
\date{\today}
\begin{document}

\maketitle

\begin{abstract}
We prove that the rank polynomial of the lattice of order ideals of a loop fence poset is unimodal. This poset arises as the poset of join-irreducibles in the lattice of good matchings of loop graphs associated with notched arcs. Equivalently, such polynomials can be obtained by evaluating all coefficient variables in an $F$-polynomial at a single variable $q$. We also conclude that the rank polynomial of any tagged arc, whether plain or notched, is not only unimodal but also satisfies a symmetry condition known as almost interlacing. Furthermore, when the lamination consists of a single curve, the cluster expansion—evaluated by setting all cluster variables to $1$ and all coefficient variables to $q$—is also unimodal. We conjecture that polynomials in this case are log-concave.
\end{abstract}

\tableofcontents

\section{Introduction}

Cluster algebras were introduced by Fomin and Zelevinsky in 2002 in their foundational work on dual canonical bases~\cite{FZ2002}. Since then, they have been found to intersect with a broad range of mathematical disciplines, including total positivity, quiver representations, Teichm\"uller theory, tropical geometry, Lie theory, and Poisson geometry. Throughout this paper, we refer to the cluster algebras defined in~\cite{FZ2002} as \emph{ordinary cluster algebras}.

A particularly significant subclass of ordinary cluster algebras consists of those associated with surfaces, known as \emph{cluster algebras from surfaces}. For these, various combinatorial frameworks have been developed to describe the cluster variables corresponding to arcs on the surface. In~\cite{musiker2011positivity}, Musiker, Schiffler, and Williams introduced a powerful combinatorial tool known as the \emph{snake graph}. When an arc $\gamma$ has only plain endpoints, the cluster variable $x_\gamma$ admits an expansion with respect to the initial cluster determined by a triangulation $T$ as a generating function over perfect matchings (or dimer coverings) of a planar bipartite graph $\mathcal{G}_{\gamma, T}$, called a snake graph. When $\gamma$ is a closed curve, the appropriate combinatorial object is a \emph{band graph}, obtained by identifying two boundary edges of the terminal tiles of the snake graph. In the case where $\gamma$ has a notched endpoint, the formula in~\cite{musiker2011positivity} uses a specific subset of perfect matchings from the snake graph associated with a related plain arc. Wilson later provided an alternative construction using \emph{loop graphs}~\cite{wilson2020surface}.

The poset of join-irreducibles in the lattice of perfect matchings of a snake graph is always a \emph{fence poset}. Fence posets form a natural class of partially ordered sets that arise not only in the study of cluster algebras but also in quiver representations and various areas of enumerative combinatorics. Their connection to cluster algebras from surfaces has been further explored in~\cite{ouguz2024cluster, pilaud2023posets}, where the authors construct a poset of join-irreducibles in the lattice of good matchings of a loop graph. These posets are referred to as \emph{loop fence posets}.

The relationship between rank polynomials and cluster algebras is mediated through the theory of $F$-polynomials. A fundamental property of cluster algebras is the \emph{Laurent phenomenon}~\cite{FZ2002}, which states that, with respect to a fixed initial cluster, any cluster variable can be expressed as a Laurent polynomial in the initial cluster variables. According to the separation formula of Fomin and Zelevinsky~\cite{FZ2007}, this Laurent polynomial can be factored into a Laurent monomial—encoded by an integer vector called the $g$-vector—and an auxiliary polynomial $F$ with integer coefficients, known as the \emph{$F$-polynomial}. These $F$-polynomials encode the combinatorics of mutation within the cluster algebra. When each coefficient variable $y_i$ in the $F$-polynomial is substituted by a single variable $q$, the resulting polynomial coincides with the rank polynomial of the associated poset.

The variables \(y_i\) in the \(F\)-polynomial are referred to as \emph{coefficient variables}. In cluster algebras from surfaces, these correspond to geometric coefficients realized via \emph{laminations} which were discovered by Fomin and Thurston~\cite{FT-II}. When the lamination consists of a single curve, it is called a \emph{single lamination}. Given an arc $\gamma$ in a triangulation \(T\) with an endpoint \(p\) tagged plain, the \emph{elementary lamination} associated to $\gamma$ is a curve \(L_\gamma\) defined as follows: it terminates at a point \(q\) near \(p\) on a boundary segment, chosen so that the path from \(p\) to \(q\) along the boundary keeps the surface \(S\) on the right. Thus, \(L_\gamma\) agrees with \(\gamma\) except near its endpoints, where it turns slightly to the right to end on a boundary segment or spiral counterclockwise into a puncture. If \(\gamma\) is tagged notched at one of its endpoints, then \(L_\gamma\) instead spirals clockwise into the puncture.

Fence posets have attracted considerable interest in recent years, particularly after a conjecture by Morier-Genoud and Ovsienko~\cite{MGO20}, which asserted that the rank polynomial of the lattice of order ideals of a fence poset---equivalently, the rank polynomial of the lattice of perfect matchings of a snake graph---is unimodal. Beyond this, the authors used these polynomials to define \( q \)-deformed rational numbers. The conjecture was resolved in~\cite{OR23}, where it was shown that the coefficient sequences of these rank polynomials are not only unimodal but also satisfy a certain inequality, called \eqref{eq:ineqA} and defined later in Section 3.1. We define a sequence that is unimodal and satisfies \eqref{eq:ineqA} as \emph{almost interlacing}.

The case of loop fence posets remained open. In this paper, we address this gap by proving that for any notched arc, the rank polynomial of the lattice of order ideals of the associated fence poset is almost interlacing. Consequently, the rank polynomial of any tagged arc---whether plain or notched---is both unimodal and almost interlacing.

Furthermore, we show that in the case of a single lamination on a punctured surface, the associated cluster expansion for an arc $\gamma$—obtained by substituting $1$ for all cluster variables and $q$ for each coefficient variable $y_i$—is also unimodal. Computational experiments support the conjecture that these polynomials are in fact log-concave. However, the standard approach via geometric lattices does not apply in this setting.

The structure of the paper is as follows. Section~\ref{sec:cluster_algebra} provides background on cluster algebras from surfaces. In Section~\ref{sec:combi}, we present the expansion formula using order ideals of fence posets. Section~\ref{sec:uniloop} establishes the unimodality of the rank polynomial in the case of loop fence posets. In Section~\ref{sec:singlelam}, we extend this result to the setting of single laminations and conjecture that log-concavity shows up.

\section{Cluster Algebras from Surfaces}\label{sec:cluster_algebra}

\subsection{Cluster Algebra}
We begin by reviewing the definition of cluster algebras, initially introduced by Fomin and Zelevinsky in \cite{FZ2002}. Our explanation follows the framework set out in \cite{FZ2007}.

Let \(n \leq m\) be positive integers. Let \(\mathcal{F}\) denote the field of rational functions in \(m\) independent variables. Consider a collection of algebraically independent variables \(x_1, \ldots, x_n, x_{n+1}, \ldots, x_m\) within \(\mathcal{F}\). The coefficient ring is defined as \(\mathbb{Z}\mathbb{P} = \mathbb{Z}[x_{n+1}, \ldots, x_m]\). 

\begin{definition}
A \emph{(labeled) seed} is a pair \((\mathbf{x}, B)\), where \(\mathbf{x} = (x_1, \ldots, x_m)\) is a set of variables in \(\mathcal{F}\) that are algebraically independent over \(\mathbb{Z}\mathbb{P}\), and \(B = (b_{jk})_{j \in \{1, 2, \ldots, m\}, k \in \{1, \ldots, n\}}\) is an \(m \times n\) integer matrix, whose sub-matrix with top $n$ rows is skew-symmetric. The variables in any seed are referred to as \emph{cluster variables}. The variables \(x_{n+1}, \ldots, x_m\) are called \emph{frozen variables}.
\end{definition}

We obtain unlabeled seeds from labeled seeds by considering the labeled seeds under equivalence. Specifically, two labeled seeds are identified if one can be transformed into the other via a simultaneous permutation of the components of the cluster $\mathbf{x}$ and the corresponding rows and columns of the exchange matrix $B$.

\begin{definition}
Let $[x]_{+} = \max(x, 0)$. Given a labeled seed $(\mathbf{x}, B)$ in $\mathcal{F}$ and an index $k \in \{1, \dots, n\}$, the \emph{seed mutation} $\mu_k(\mathbf{x}, B)$ in direction $k$ produces a new seed $(\mathbf{x}', B')$ according to the following rules:
\begin{itemize}
    \item The entries $b'_{ij}$ of the mutated exchange matrix $B'$ are given by:
    \[
    b'_{ij} = 
    \begin{cases} 
    -b_{ij} & \text{if } i=k \text{ or } j=k \\ 
    b_{ij} + \frac{|b_{ik}|b_{kj} + b_{ik}|b_{kj}|}{2} & \text{otherwise} 
    \end{cases}
    \]
    (Equivalently: $b'_{ij} = b_{ij} + [-b_{ik}]_{+}b_{kj} + b_{ik}[b_{kj}]_{+}$)
    
    \item The new cluster is $\mathbf{x}' = (x_1, \dots, x_{k-1}, x'_k, x_{k+1}, \dots, x_n)$, where the new cluster variable $x'_k$ is defined by the \emph{exchange relation}:
    \[
    x'_k = \frac{1}{x_k} \left( \prod_{b_{ik} > 0} x_i^{b_{ik}} + \prod_{b_{ik} < 0} x_i^{-b_{ik}} \right)
    \]
\end{itemize}
\end{definition}
\begin{definition}
    Let $\mathbb{T}_n$ represent an $n$-regular tree, where the edges are labeled with numbers $1,\ldots,n$, ensuring that each vertex's $n$ edges have distinct labels. A \emph{seed pattern} is created by associating a labeled seed $\Sigma_t = (\tilde{\x}(t), \tilde{B}(t))$ with each vertex $t \in \mathbb{T}_n$. The seeds assigned to the vertices at the endpoints of any edge \( t \overset{k}{\rule[0.5ex]{2em}{0.5pt}} t' \) are related through a seed mutation in the $k$ direction. A seed pattern is fully determined by any one of its seeds.
\end{definition}

\begin{definition}
    Let $(\tilde{\x}(t), \tilde{B}(t))_{t \in \mathbb{T}_n}$ represent a seed pattern as described earlier. The set of all \emph{cluster variables} appearing in these seeds is defined as
    \[\mathcal{X} = \bigcup_{t \in \mathbb{T}_n} \x(t) = \{x_{i,t} : t \in \mathbb{T}_n, 1 \leq i \leq n\}.\]
    We define the \emph{ground ring} to be $R = \Z[x_{n+1}, \ldots, x_m]$, which is the polynomial ring generated by the frozen variables. The \emph{cluster algebra} $\mathcal{A}$ (of geometric type, over $R$) associated with this seed pattern is the $R$-subalgebra of the ambient field $F$ generated by all cluster variables: $\mathcal{A} = R[\mathcal{X}]$. Specifically, a cluster algebra is the $R$-subalgebra $\mathcal{A}$ as described, along with a fixed seed pattern within it.
\end{definition}

Two fundamental properties of cluster algebras are the \emph{Laurent phenomenon}, which guarantees that every cluster variable is a \emph{Laurent polynomial}, and \emph{positivity}, which asserts that all coefficients in these polynomials are non-negative integers.

\begin{theorem}[Theorem 3.1 of~\cite{FZ2002}]
    Every element of $\mathcal{A}$ is a Laurent polynomial over $R$ in the cluster variables from $\x$.
\end{theorem}

\begin{theorem}[Corollary 0.4 of~\cite{GHKK18}]
    Every coefficient appearing in any cluster variable is a non-negative integer.
\end{theorem}

\subsection{Constructing Cluster Algebra from Surfaces}
Let \( S \) be a connected, oriented 2-dimensional Riemann surface with a possibly empty boundary, and let \( M \) be a finite set of marked points on \( S \) such that there is at least one marked point on each boundary component. For technical reasons, we exclude the following cases for the pair $(S, M)$: 
\begin{itemize}
    \item A sphere with fewer than four punctures;
    \item A monogon with fewer than two punctures;
    \item An unpunctured disc with fewer than two marked points on its boundary;
    \item A closed surface with fewer than three punctures (specifically, we exclude those with only two).
\end{itemize}

We start with a brief review of some relevant concepts related to \emph{cluster algebras from surfaces}. We adopt the standard definition of a cluster algebra from \cite{FZ2002} and follow the surface model proposed by Fomin, Shapiro, and D. Thurston in \cite{FST-I}. While we will outline a few necessary definitions, for a more comprehensive discussion, we refer readers to \cite{FST-I} or to expository works such as \cite{glick2017introduction, williams2014cluster}. In this paper, we focus on the subclass of cluster algebras that can be represented by triangulations of marked surfaces, particularly emphasizing \emph{punctured surfaces} where marked points can appear in the interior of the surface. These interior marked points are known as \emph{punctures}.

\begin{definition}
An \emph{arc} $\gamma$ on a surface $(S,M)$ is a non-self-intersecting curve within \(S\) with endpoints in \(M\), but otherwise disjoint from \(M\) and the boundary $\partial S$. Arcs cannot be contractible to the boundary $\partial S$, nor can they cut out an unpunctured monogon or bigon. Arcs are considered up to isotopy. Each arc is oriented, with \(s(\gamma)\) denoting the \emph{starting point} and \(t(\gamma)\) denoting the \emph{terminal point}.
\end{definition}

An \emph{ideal triangulation} of $(S,M)$ is a maximal set of pairwise compatible plain arcs, which decompose $(S,M)$ into triangles. These triangles may share sides; two sides of the same triangle can be glued together, forming a \emph{self-folded triangle}. Since an arc enclosed by a self-folded triangle cannot be flipped, to allow mutation at every cluster variable (or, equivalently, to make all arcs flippable), punctured surfaces require the more nuanced concept of \emph{tagged arcs}.

\begin{definition}
A \emph{tagged arc} is derived from an ordinary arc by assigning each of its endpoints a tag, either \emph{plain} or \emph{notched}. In diagrams, notched endpoints are marked with the $\bowtie$ symbol. Endpoints on the boundary $\partial S$ must be tagged plain, and both endpoints of loops must have identical tagging.
\end{definition}

A tagged triangulation is a maximal collection of pairwise compatible tagged arcs on $(S,M)$, where loops that enclose a once-punctured monogon are not allowed. The compatibility of tagged arcs is defined as follows:

\begin{definition}
For a given tagged arc $\gamma$, let $\gamma^0$ represent the underlying plain arc. Tagged arcs $\alpha$ and $\beta$ in $(S,M)$ are considered \emph{compatible} if and only if the following conditions are met:
\begin{itemize}
    \item The isotopy classes of $\alpha^0$ and $\beta^0$ have non-intersecting representatives;
    \item If $\alpha^0 \neq \beta^0$ but $\alpha$ and $\beta$ share an endpoint $p$, then both arcs must have the same tagging at $p$;
    \item If $\alpha^0 = \beta^0$, then at least one end of $\alpha$ must have the same tagging as the corresponding end of $\beta$.
\end{itemize}    
\end{definition}

In this paper, we consider only tagged triangulations that do not contain self-folded triangles. Consequently, (1) our triangulations will only include plain arcs, and (2) for each puncture \(p\) on the marked surface, there will be at least two distinct arcs from the triangulation incident to \(p\). These arcs are often referred to as \emph{spokes}.

Note that for each ideal triangulation \(T\), we can uniquely assign a tagged triangulation \(\iota(T)\) by replacing a loop that encloses a once-punctured monogon with a singly notched arc, with the notch incident to the enclosed puncture.

\begin{remark}\label{re:incomp_tag_plain}
The compatibility between tagged arcs is invariant under a simultaneous change all tags at a specific puncture. Consequently, for any tagged triangulation $T$, one may perform such a transformation at every puncture where all incident arc ends are notched. This procedure yields a modified tagged triangulation $\hat{T}$ that is the image of an ideal triangulation $T^0$ (possibly containing self-folded triangles) under the mapping $\iota$; $\hat{T} = \iota(T^0)$.
\end{remark}

When considering the intersection of arcs, we always use representatives of the isotopy classes with the minimum number of crossings. For convenience, let \(e(\gamma,\gamma')\) denote the minimum number of crossings between representatives of the isotopy classes of two arcs \(\gamma\) and \(\gamma'\). The first condition from the previous definition can then be restated as requiring \(e(\gamma,\gamma') = 0\) for \(\gamma\) and \(\gamma'\) to be compatible.

In the standard surface model framework, a (possibly punctured) surface \((S,M)\) defines a surface-type cluster algebra \(\mathcal{A}\) with the following correspondences: clusters of \(\mathcal{A}\) correspond to tagged triangulations of \((S,M)\); individual tagged arcs in a triangulation \(T\) correspond to cluster variables of \(\mathcal{A}\), and mutations in \(\mathcal{A}\) correspond to flips of tagged arcs.

For convenience, we often represent tagged arcs using \emph{hooks}. This concept was used by Wilson in \cite{wilson2020surface} to construct \emph{loop graphs} and to derive cluster expansion formulas for tagged arcs, and by Labardini-Fragoso as well as Domínguez to define modules associated with tagged arcs in the Jacobian algebra from a punctured surface \cite{dominguez2017arc}.

\begin{definition}
Let \(T\) be an ideal triangulation of a surface \(S\) and \(\gamma\) be a directed arc on \(S\) with an endpoint at a puncture \(p\). A \emph{hook} at \(p\) is a curve that either:
\begin{itemize}
    \item winds around \(p\) clockwise or counterclockwise, intersecting all spokes incident to \(p\) exactly once, and then follows \(\gamma\), if \(\gamma\) begins at \(p\),
    \item or follows \(\gamma\) and then winds around \(p\) clockwise or counterclockwise, intersecting all spokes incident to \(p\) exactly once, if \(\gamma\) ends at \(p\).
\end{itemize}
\end{definition}
Figure~\ref{fig:hook} illustrates a tagged arc $\gamma^{(p)}$ together with its two possible hook replacements at the notched endpoint.

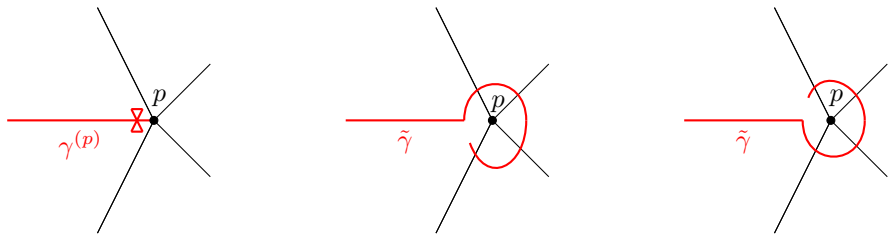
\begin{figure}[H]
    \centering
    \begin{tikzpicture}[scale=1.5]
    \draw (-2,0) to (-2.5,1);
    \draw (-2,0) to (-2.5,-1);
    \draw (-2.5,1) to (-2,0);
    \draw (-2.5,-1) to (-2,0);
    \draw (-2,0) to (-1.5,0.5);
    \draw (-2,0) to (-1.5,-0.5);
    \draw[red,thick] (-3.3,0) to node[midway,below]{$\gamma^{(p)}$} (-2,0);
    \draw[red,thick] (-2.1,0.1) to (-2.2,-0.1);
    \draw[red,thick] (-2.2,0.1) to (-2.1,-0.1);
    \draw[red,thick] (-2.1,0.1) to (-2.2,0.1);
    \draw[red,thick] (-2.1,-0.1) to (-2.2,-0.1);
    \draw[fill=black] (-2,0) circle [radius=1pt];
    \node at (-1.95,0.2) {$p$};

    \draw (1,0) to (0.5,1);
    \draw (1,0) to (0.5,-1);
    \draw (0.5,1) to (1,0);
    \draw (0.5,-1) to (1,0);
    \draw (1,0) to (1.5,0.5);
    \draw (1,0) to (1.5,-0.5);
    \draw[red,thick] (-0.3,0) to node[midway,below]{$\tilde{\gamma}$} (0.75,0);
    \draw[red, thick, out=90,in=90,looseness=2] (0.75,0) to (1.3,0);
    \draw[red, thick, out=-90,in=-70,looseness=2] (1.3,0) to (0.8,-0.2);
    \draw[fill=black] (1,0) circle [radius=1pt];
    \node at (1.05,0.15) {$p$};

    \draw (4,0) to (3.5,1);
    \draw (4,0) to (3.5,-1);
    \draw (3.5,1) to (4,0);
    \draw (3.5,-1) to (4,0);
    \draw (4,0) to (4.5,0.5);
    \draw (4,0) to (4.5,-0.5);
    \draw[red,thick] (2.7,0) to node[midway,below]{$\tilde{\gamma}$} (3.75,0);
    \draw[red, thick, out=-90,in=-90,looseness=2] (3.75,0) to (4.3,0);
    \draw[red, thick, out=90,in=70,looseness=1.5] (4.3,0) to (3.8,0.2);
    \draw[fill=black] (4,0) circle [radius=1pt];
    \node at (4.05,0.2) {$p$};

    \end{tikzpicture}
    \caption{A tagged arc $\gamma^{(p)}$ and its two hook replacements at the notched endpoint.}
    \label{fig:hook}
\end{figure}

Also, we will consider several other families of curves on a surface here, i.e. closed curves and arcs with self-intersection. The latter are sometimes referred in other literatures as \emph{generalized arcs}. If a closed curve has no self-intersections, it is called \emph{essential}.

Following the construction in \cite{FST-I}, one may associate an exchange matrix—and consequently a cluster algebra—with any bordered surface $(S, M)$.

\begin{definition}
Let $T$ be an ideal triangulation with arcs $\{\tau_1, \ldots, \tau_n\}$. For each triangle $\Delta$ in $T$ that is not self-folded, we define a local contribution matrix $B^{\Delta} = (b^{\Delta}_{ij})$ as follows:
\begin{itemize}
    \item The non-zero entries $b^{\Delta}_{ij} \in \{1, -1\}$ are assigned based on the following configurations:
    \begin{enumerate}[label=(\alph*)]
        \item If $\tau_i$ and $\tau_j$ are sides of $\Delta$, then $b^{\Delta}_{ij} = 1$ (and $b^{\Delta}_{ji} = -1$) if $\tau_j$ immediately follows $\tau_i$ in clockwise order.
        \item If $\tau_j$ is the radius of a self-folded triangle enclosed by a loop $\tau_l$, and $\tau_i, \tau_l$ are sides of $\Delta$, then $b^{\Delta}_{ij} = 1$ if $\tau_l$ follows $\tau_i$ clockwisely.
        \item If $\tau_i$ is the radius of a self-folded triangle enclosed by a loop $\tau_l$, and $\tau_l, \tau_j$ are sides of $\Delta$, then $b^{\Delta}_{ij} = 1$ if $\tau_j$ follows $\tau_l$ clockwisely.
    \end{enumerate}
    \item In all other instances, $b^{\Delta}_{ij} = 0$.
\end{itemize}

The \emph{signed adjacency matrix} $B_T = (b_{ij})$ associated with $T$ is obtained by summing these contributions over all non-self-folded triangles: $b_{ij} = \sum_{\Delta} b^{\Delta}_{ij}$.
\end{definition}

\begin{example}
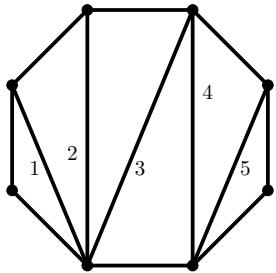
\begin{figure}[H]
\centering
\begin{tikzpicture}
    \draw[line width=0.5mm] (0,0) to (1.4,0);
    \draw[line width=0.5mm] (1.4,0) to (2.4,1);
    \draw[line width=0.5mm] (2.4,2.4) to (2.4,1);
    \draw[line width=0.5mm] (2.4,2.4) to (1.4,3.4);
    \draw[line width=0.5mm] (0,3.4) to (1.4,3.4);
    \draw[line width=0.5mm] (0,3.4) to (-1,2.4);
    \draw[line width=0.5mm] (-1,1) to (0,0);
    \draw[line width=0.5mm] (-1,1) to (-1,2.4);

    \draw[line width=0.5mm] (1.4,0) to (2.4,2.4);
    \draw[line width=0.5mm] (1.4,3.4) to (1.4,0);
    \draw[line width=0.5mm] (1.4,3.4) to (0,0);
    \draw[line width=0.5mm] (0,3.4) to (0,0);
    \draw[line width=0.5mm] (-1,2.4) to (0,0);

    \node[scale=0.8] at (-0.7,1.3) {$1$};
    \node[scale=0.8] at (-0.2,1.5) {$2$};
    \node[scale=0.8] at (0.7,1.3) {$3$};
    \node[scale=0.8] at (1.6,2.3) {$4$};
    \node[scale=0.8] at (2.1,1.3) {$5$};

    \filldraw[black] (0,0) circle (2pt);
    \filldraw[black] (1.4,0) circle (2pt);
    \filldraw[black] (2.4,1) circle (2pt);
    \filldraw[black] (2.4,2.4) circle (2pt);
    \filldraw[black] (1.4,3.4) circle (2pt);
    \filldraw[black] (0,3.4) circle (2pt);
    \filldraw[black] (-1,2.4) circle (2pt);
    \filldraw[black] (-1,1) circle (2pt);

\end{tikzpicture}
\caption{Example of an ideal triangulation $T$ of octagon}
    \label{fig:matrix_example}
\end{figure}

Consider the example in Figure~\ref{fig:matrix_example}. It is an ideal triangulation $T$ of an octagon. one can check that the signed adjacency matrix of this triangulation $T$ is
\[B(T)=\begin{bmatrix}
0 & -1 & 0 & 0 & 0\\
1 & 0 & -1 & 0 & 0\\
0 & 1 & 0 & 1 & 0\\
0 & 0 & -1 & 0 & -1\\
0 & 0 & 0 & 1 & 0
\end{bmatrix}.	
\]
\end{example}

For a tagged triangulation $T$, the signed adjacency matrix $B_T$ is defined by the matrix $B_{T^0}$, where $T^0$ is the ideal triangulation associated with $T$ via the transformation detailed in Remark~\ref{re:incomp_tag_plain}. The index sets for these matrices, which correspond to the sets of tagged and ideal arcs respectively, are identified through a canonical bijection.

\begin{theorem}[Theorem 6.1 of~\cite{FT-II}]
Let $(S, M)$ denote a surface with marked points, and consider an initial triangulation of the surface $S$. A cluster algebra $\mathcal{A}$ associated to the signed adjacency matrix of a tagged triangulation possesses the following characteristics:
\begin{itemize}
    \item There exists a bijection between the seeds and the tagged triangulations of $(S, M)$; and
    \item There exists a bijection between cluster variables in $\mathcal{A}$ and tagged arcs in $(S, M)$.
\end{itemize}
\end{theorem}
\subsection{Lamination and Shear Coordinates}
To construct frozen variables in cluster algebra of surface type, we introduce the notion of lamination and shear coordinate~\cite{Fock2007,FT-II}. 

\begin{definition}
A \emph{lamination} on a bordered surface $(S, M)$ is a collection of finitely many simple, mutually non-intersecting curves in $S$, considered up to isotopy relative to $M$. Each curve must belong to one of the following classes:
\begin{itemize}
    \item A closed curve;
    \item A curve connecting two unmarked boundary points;
    \item A curve connecting an unmarked boundary point to a puncture, into which it spirals (either clockwise or counter-clockwise);
    \item A curve whose both ends spiral into punctures (not necessarily distinct).
\end{itemize}
The following curves are excluded:
\begin{itemize}
    \item Curves that contract to a point or enclose a single puncture (i.e., curves bounding a disk with $\le 1$ puncture);
    \item Curves with boundary endpoints that are isotopic to a segment of the boundary containing at most one marked point.
\end{itemize}
A lamination consisting of exactly one such curve is referred to as a \emph{single lamination}.
\end{definition}

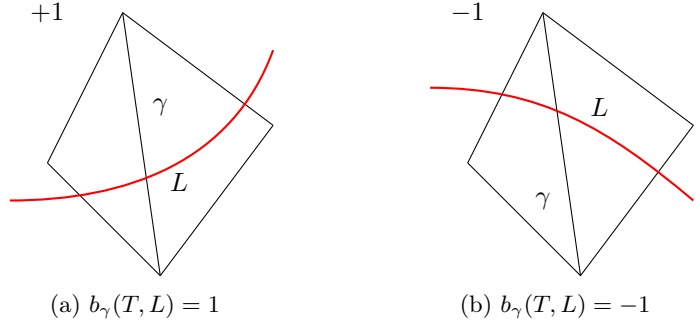
\begin{figure}[H]
    \centering
     \begin{subfigure}[b]{0.2\textwidth}
        \centering
\begin{tikzpicture}[scale=0.5, transform shape]

\draw (-3,-1) to (0,-4);
\draw (0,-4) to (3,0);
\draw (0,-4) to (-1,3);
\draw (3,0) to (-1,3);
\draw (-3,-1) to (-1,3);

\draw[out=0,in=-110,looseness=1,thick, red] (-4,-2) to (3,2);

\node[scale=2] at (0,0.5) {$\gamma$};
\node[scale=2] at (0.5,-1.5) {$L$};
\node[scale=2] at (-3,3) {$+1$};

\end{tikzpicture}
        \caption{$b_{\gamma}(T,L)=1$}
        \label{fig:shearplus}
    \end{subfigure}
    \hspace{2cm}
    \begin{subfigure}[b]{0.2\textwidth}
        \centering
\begin{tikzpicture}[scale=0.5, transform shape]

\draw (-3,-1) to (0,-4);
\draw (0,-4) to (3,0);
\draw (0,-4) to (-1,3);
\draw (3,0) to (-1,3);
\draw (-3,-1) to (-1,3);

\draw[out=0,in=140,looseness=1,thick,red] (-4,1) to (3,-2);

\node[scale=2] at (-1,-2) {$\gamma$};
\node[scale=2] at (0.5,0.5) {$L$};
\node[scale=2] at (-3,3) {$-1$};

\end{tikzpicture}
        \caption{$b_{\gamma}(T,L)=-1$}
        \label{fig:shearmin}
    \end{subfigure}
    \caption{Defining the shear coordinates}
    \label{fig:shear}
\end{figure}

\begin{definition}
Let $L$ be a lamination and $T$ a triangulation containing no self-folded triangles. For each arc $\gamma \in T$, the shear coordinate $b_\gamma(T, L)$ is the sum of local contributions from the intersections of $L$ with $\gamma$. A segment of a curve in $L$ crossing the quadrilateral surrounding $\gamma$ contributes $+1$ if it forms an $S$-shape and $-1$ if it forms a $Z$-shape, as depicted in Figure~\ref{fig:shear}. These two configurations are mutually exclusive for a given arc. Furthermore, while a spiraling curve may technically intersect an arc infinitely many times, only a finite subset of these intersections yield non-zero contributions to $b_\gamma(T, L)$.
\label{def:shear}
\end{definition}

\begin{theorem}[Theorem 13.6 of~\cite{FT-II}]
For a fixed triangulation $T$ without self-folded triangles, the map 
\[ L \mapsto (b_{\gamma}(T, L))_{\gamma \in T} \]
defines a bijection between the set of integral unbounded measured laminations and the integer lattice $\mathbb{Z}^n$.
\end{theorem}

\begin{definition}
A \emph{multi-lamination} is a finite collection of laminations $\mathbf{L} = (L_{n+1}, \ldots, L_{m})$. Given such a collection and an ideal triangulation $T$, the extended exchange matrix $\tilde{B} = \tilde{B}(T, \mathbf{L}) = (b_{ij})$ is constructed as follows: the principal $n \times n$ submatrix is the signed adjacency matrix $B(T)$, and the remaining $m-n$ rows are populated by the shear coordinates of the constituent laminations:
\[
b_{n+i, \gamma} = b_{\gamma}(T, L_{n+i}) \quad \text{for } 1 \leq i \leq m-n.
\]
\end{definition}

As established in \cite{FT-II}, the extended matrix $\tilde{B}(T, \mathbf{L})$ is compatible with the mutation of the triangulation $T$.

\begin{example}\label{ex:single}

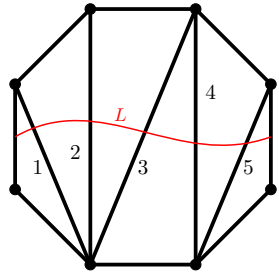
\begin{figure}[H]
    \centering
\begin{tikzpicture}

    \draw[line width=0.5mm] (0,0) to (1.4,0);
    \draw[line width=0.5mm] (1.4,0) to (2.4,1);
    \draw[line width=0.5mm] (2.4,2.4) to (2.4,1);
    \draw[line width=0.5mm] (2.4,2.4) to (1.4,3.4);
    \draw[line width=0.5mm] (0,3.4) to (1.4,3.4);
    \draw[line width=0.5mm] (0,3.4) to (-1,2.4);
    \draw[line width=0.5mm] (-1,1) to (0,0);
    \draw[line width=0.5mm] (-1,1) to (-1,2.4);

    \draw[line width=0.5mm] (1.4,0) to (2.4,2.4);
    \draw[line width=0.5mm] (1.4,3.4) to (1.4,0);
    \draw[line width=0.5mm] (1.4,3.4) to (0,0);
    \draw[line width=0.5mm] (0,3.4) to (0,0);
    \draw[line width=0.5mm] (-1,2.4) to (0,0);

    \node[scale=0.8] at (-0.7,1.3) {$1$};
    \node[scale=0.8] at (-0.2,1.5) {$2$};
    \node[scale=0.8] at (0.7,1.3) {$3$};
    \node[scale=0.8] at (1.6,2.3) {$4$};
    \node[scale=0.8] at (2.1,1.3) {$5$};

    \draw[line width=0.2mm, out=30,in=200,looseness=1, red] (-1,1.7) to (2.4,1.7);

    \node[scale=0.7, red] at (0.4,2) {$L$};

    \filldraw[black] (0,0) circle (2pt);
    \filldraw[black] (1.4,0) circle (2pt);
    \filldraw[black] (2.4,1) circle (2pt);
    \filldraw[black] (2.4,2.4) circle (2pt);
    \filldraw[black] (1.4,3.4) circle (2pt);
    \filldraw[black] (0,3.4) circle (2pt);
    \filldraw[black] (-1,2.4) circle (2pt);
    \filldraw[black] (-1,1) circle (2pt);

\end{tikzpicture}
    \caption{Example of lamination $L$ in triangulation $T$}
    \label{fig:lamination}
\end{figure}
    Figure~\ref{fig:lamination} shows an example of a lamination $L$ in a triangulation $T$ in Figure~\ref{fig:matrix_example} above. Its matrix $\tilde{B}(T)$ is the following:
    \[\tilde{B}(T)=\begin{bmatrix}
    0 & -1 & 0 & 0 & 0\\
    1 & 0 & -1 & 0 & 0\\
    0 & 1 & 0 & 1 & 0\\
    0 & 0 & -1 & 0 & -1\\
    0 & 0 & 0 & 1 & 0\\
    -1 & 0 & 1 & -1 & 1
    \end{bmatrix}.	
    \]
\end{example}

\begin{definition}
    For a given tagged arc $\gamma$ in a marked surface $(S,M)$, we construct its \emph{elementary lamination} $L_{\gamma}$ as a curve that closely follows $\gamma$ within a small tubular neighborhood. The construction of $L_{\gamma}$ is contingent upon the termination of $\gamma$:
    \begin{itemize}
    \item If $\gamma$ terminates at a boundary component $C$ of $S$ at a point $a$, then $L_{\gamma}$ begins at a point $a' \in C$ located counterclockwise to $a$, and subsequently follows the path of $\gamma$.
    \item If $\gamma$ terminates at a puncture, $L_{\gamma}$ spirals around it: in a counterclockwise direction if $\gamma$ has a plain tag at the puncture, and in a clockwise direction if the tag is notched.
    \end{itemize}
\end{definition}

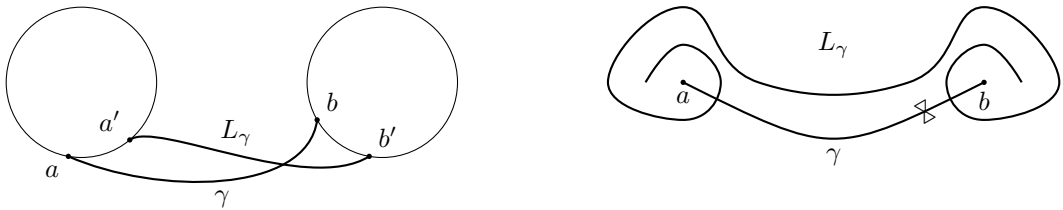
\begin{figure}[H]
    \centering
    \begin{tikzpicture}
    \draw (0,0) circle (1cm) node[below] {};
    \coordinate (P1) at (1*cos{260}, 1*sin{260});
    \coordinate (P2) at (1*cos{310}, 1*sin{310});
    \fill (P1) circle (1pt) node[below left] {$a$};
    \fill (P2) circle (1pt) node[above left] {$a^\prime$};

    \draw (4,0) circle (1cm) node[below] {};
    \coordinate (Q1) at (4+1*cos{-150}, 1*sin{-150});
    \coordinate (Q2) at (4+1*cos{-100}, 1*sin{-100});
    \fill (Q1) circle (1pt) node[above right] {$b$};
    \fill (Q2) circle (1pt) node[above right] {$b^\prime$};

    \draw[thick] (P1) .. controls (1, -1.5) and (3, -1.5) .. (Q1)
    node[midway, below] {$\gamma$};

    \draw[thick] (P2) .. controls (1, -0.5) and (3, -1.5) .. (Q2)
    node[midway, above] {$L_{\gamma}$};
    
    \coordinate (R1) at (8, 0);
    \coordinate (R2) at (12, 0);

    \fill (R1) circle (1pt) node[below] {$a$};
    \fill (R2) circle (1pt) node[below] {$b$};

    \draw[thick] (R1) .. controls (10, -1) and (10, -1) .. (R2)
    node[midway, below, sloped] {$\gamma$}
    node[midway, pos=.85, rotate=120]{$\bowtie$}(12, 0);

    \coordinate (S1) at (7.5, 0);
    \coordinate (S2) at (8, 0.5);
    \coordinate (S3) at (8.5, 0);
    \coordinate (S4) at (8, -0.5);
    \coordinate (S5) at (7, 0);
    \coordinate (S6) at (8, 1);
    \coordinate (S7) at (9, 0);
    \coordinate (S8) at (11, 0);
    \coordinate (S9) at (12, 1);
    \coordinate (S10) at (13, 0);
    \coordinate (S11) at (12, -0.5);
    \coordinate (S12) at (11.5, 0);
    \coordinate (S13) at (12, 0.5);
    \coordinate (S14) at (12.5, 0);

    \draw[thick] plot[smooth, tension=0.8] coordinates {(S1) (S2) (S3) (S4) (S5) (S6) (S7) (S8) (S9) (S10) (S11) (S12) (S13) (S14)};
    \node at (10, 0.5) {$L_{\gamma}$};
        
    \end{tikzpicture}
    \caption{An elementary lamination $L_{\gamma}$ associated with a tagged arc $\gamma$, plain(left) and notched(right).}
    \label{fig:elementary_lamination}

\end{figure}

\begin{definition}\label{def:principal_coefficients}
A cluster pattern $t \mapsto\left(\mathbf{x}_t, \mathbf{y}_t, B_t\right)$ on $\mathbb{T}_n$ (or the corresponding cluster algebra $\mathcal{A}$) is said to have \emph{principal coefficients} at a vertex $t_0$ if the semifield $\mathbb{P}$ is equal to the tropical semifield generated by the coefficients $y_1, \ldots, y_n$, called $y$-variables, and the coefficient tuple $\mathbf{y}_{t_0}$ is equal to $(y_1, \ldots, y_n)$. In this case, we denote the cluster algebra by $\mathcal{A}=\mathcal{A}_{\bullet}\left(B_{t_0}\right)$.
\end{definition}

\begin{remark}
Definition~\ref{def:principal_coefficients} can be rephrased as follows: a cluster algebra $\mathcal{A}$ have principal coefficients at vertex $t_0$ if it is of geometric type and its exchange matrix $\tilde{B}_{t_0}$ is of size $2n \times n$, such that the upper $n \times n$ submatrix is $B_{t_0}$ and the lower $n \times n$ submatrix is the identity matrix $I_n$ (cf. \cite{FZ2002}).
\end{remark}

\begin{definition}
Consider a cluster algebra $\mathcal{A}$ with principal coefficients at $t_0$, determined by the initial seed $\Sigma_{t_0}=\left(\mathbf{x}_{t_0}, \mathbf{y}_{t_0}, B_{t_0}\right)$, where
\[
\mathbf{x}_{t_0}=\left(x_1, \ldots, x_n\right), \quad \mathbf{y}_{t_0}=\left(y_1, \ldots, y_n\right), \quad B_{t_0}=B^0=\left(b_{i j}^0\right).
\]
By the Laurent phenomenon, each cluster variable $x_{\ell ; t}$ can be expressed as a unique Laurent polynomial in the variables $x_1, \ldots, x_n$ and $y_1, \ldots, y_n$; we denote this expression by $X_{\ell ; t} = X_{\ell ; t}^{B^0 ; t_0}$.

The \emph{$F$-polynomial}, denoted $F_{\ell ; t}(y_1, \ldots, y_n)$, is obtained by specializing all $x_i$ variables to 1:
\[
F_{\ell ; t}\left(y_1, \ldots, y_n\right)=X_{\ell ; t}\left(1, \ldots, 1 ; y_1, \ldots, y_n\right).
\]
As shown in \cite{FZ2007}, $F_{\ell ; t}$ is indeed a polynomial in the variables $y_j$.
\end{definition}

Define a $\mathbb{Z}^n$-grading on $\mathcal{A}$ by a degree homomorphism $\deg:\mathcal{A} \to \mathbb{Z}^n$, where $\deg(x_i) = \mathbf{e}_i$ (the $i$-th standard basis vector) and $\deg(y_i) = -\mathbf{b}_i$ (the $i$-th column of $B_{t_0}$).

\begin{definition}
The \emph{\textbf{g}-vector} $\mathbf{g}_{\ell;t} \in \mathbb{Z}^n$ is the degree of the cluster variable $X_{\ell;t}$ under this grading:
\[
\mathbf{g}_{\ell;t} = \deg(X_{\ell;t}).
\]
\end{definition}

The cluster expansion in an algebra with principal coefficients serves as a standard for an algebra with arbitrary coefficients. Let $F$ be a subtraction-free rational expression in several variables. For any semifield $\mathbb{P}$ and elements $u_1, \ldots, u_r \in \mathbb{P}$, we denote by $F|_{\mathbb{P}}(u_1, \ldots, u_r)$ the evaluation of $F$ within the semifield $\mathbb{P}$.

\begin{theorem}[Theorem 3.7 of~\cite{FZ2007}]
For an arbitrary coefficient semifield $\mathbb{P}$, the cluster variables are given by the formula:
\[
x_{j;t} = \frac{X_{j;t}(x_1, \ldots, x_n; y_1, \ldots, y_n)}{F_{j;t}|_{\mathbb{P}}(y_1, \ldots, y_n)}.
\]
\end{theorem}

\section{Combinatorial Expansion Formulas}\label{sec:combi}

\subsection{Posets}

Fence posets are a class of posets that arise naturally in various areas of mathematics, including cluster algebras, quiver representations, and enumerative combinatorics. They are closely related to the study of snake diagrams, which are used to represent the combinatorial structure of cluster algebras, and are also connected to other algebraic and geometric structures such as triangulations of polygons. 

\begin{definition}
    Given a composition \(\alpha = (\alpha_1, \alpha_2, \ldots, \alpha_s)\) of a natural number \(n\), the \emph{fence poset} corresponding to \(\alpha\), denoted as \(P(\alpha)\), is a poset constructed on the ground set \(\{v_1, v_2, \ldots, v_{n+1}\}\) with a specific set of order relations. The order relations in \(P(\alpha)\) are defined as follows: 
    \[
    v_1 \prec v_2 \prec \dots \prec v_{\alpha_1} \prec v_{\alpha_1+1} \succ v_{\alpha_1+2} \succ \dots \succ v_{\alpha_1+\alpha_2+1} \prec v_{\alpha_1+\alpha_2+2} \prec \dots
    \]
    These relations describe a zigzag structure alternating between ascending and descending chains, with \( n = \alpha_1 + \dots + \alpha_s \) determining the total number of covering relations. We call $s$ as the number of parts of a composition $\alpha$.
    Note that the first part $\alpha_1$ of the composition $\alpha$ can be zero, indicating that the fence poset does not necessarily start with an increasing relation $x_1 \prec x_2$.
\end{definition}

Let \( P \) be a finite poset. A subset \( I \subseteq P\) is called a \emph{lower order ideal} (resp.\ \emph{upper order ideal}) if, whenever \( a \in I \), then \( b \prec a \) (resp.\ \( b \succ a \)) implies \( b \in I \). Unless otherwise specified, we use the term \emph{ideal} to refer to a lower order ideal, denoted \( I \trianglelefteq P \). The collection of ideals (or upper order ideals), ordered by inclusion, forms a distributive lattice \( J(P) \), ranked by the cardinality of the ideals.

For a fence poset \( P(\alpha) \) associated to $\alpha$, we denote the associated lattice by \( J(\alpha) := J(P(\alpha)) \). This lattice is graded by the sizes of its ideals, and its generating function is the \emph{rank polynomial}:
\[
R(P;q)=\sum_{I \in J(P)} q^{|I|}.
\]
We will denote $R(\alpha;q):=R(P(\alpha);q)$. The corresponding \emph{rank sequence}, denoted \( r(\alpha) \), consists of the coefficients of \( q \) in \( R(\alpha; q) \). Note that this is same as applying $q$ to all $y$-variables and applying $1$ to all cluster variables $x_i$.

A sequence \( (a_0, a_1, \ldots, a_{n+1}) \) is said to be \emph{unimodal} if there exists an index \( 0 \leq k \leq n+1 \) such that
\[
a_0 \leq a_1 \leq \dots \leq a_{k-1} \leq a_k \geq a_{k+1} \geq \dots \geq a_{n+1}.
\]

A sequence is said to exhibit \emph{top interlacing} if
\[
a_0 \leq a_{n+1} \leq a_1 \leq a_{n} \leq \dots \leq a_{\lceil {n+1}/2\rceil},
\]
and \emph{bottom interlacing} if
\[
a_{n+1} \leq a_0 \leq a_{n} \leq a_1 \leq \dots \leq a_{\lfloor {n+1}/2\rfloor}.
\]
It is called \emph{symmetric} if
\[
a_i = a_{n+1-i} \quad \text{for all } 0 \leq i \leq n+1.
\]

A polynomial is said to be \emph{unimodal} if the sequence of its coefficients is unimodal. The unimodality of the rank sequences of fence posets was first conjectured in~\cite{MGO20} and subsequently proven in~\cite{OR23}. In particular, certain structured sequences—such as those arising from top interlacing, bottom interlacing, and symmetric configurations—satisfy the following inequality, referred to as \eqref{eq:ineqA} in~\cite{OR23}:

\begin{equation*}
a_{0} \leq a_{n},\quad a_{1} \leq a_{n-1},\quad \ldots \qquad\qquad a_{n+1} \leq a_{1},\quad a_{n} \leq a_{2},\quad \ldots \label{eq:ineqA}\tag{Ineq A}
\end{equation*}

\begin{definition}
A sequence that is unimodal and satisfies the inequality \eqref{eq:ineqA} is called \emph{almost interlacing}.
\end{definition}

\begin{theorem}[Theorem 1.2 and Theorem 1.3 of   \cite{OR23}]\label{thm:rankuni}
    Let \( R(\alpha; q) = \sum_{i=0}^{n+1} a_i q^i \) be the rank polynomial of a fence poset determined by \( \alpha = (\alpha_1, \ldots, \alpha_s) \), with \( n = \alpha_1 + \dots + \alpha_s \). Then \( R(\alpha; q) \) is almost interlacing. Furthermore, the rank sequence is either top interlacing, bottom interlacing, or symmetric, depending on the parity of \( s \) and the relative values of \( \alpha_1 \) and \( \alpha_s \).
\end{theorem}

Theorem~\ref{thm:rankuni} implies that for any rank polynomial, the coefficients \( a_{\lfloor (n+1)/2 \rfloor} \) and \( a_{\lceil (n+1)/2 \rceil} \) always attain the maximum value among all coefficients \( a_i \). In addition, the authors of~\cite{OR23} considered \emph{circular fence posets}, which are obtained by imposing the additional relation \( v_{n+1} = v_1 \) on the standard fence poset. We denote such posets as \( \overline{P}(\alpha) \), the corresponding rank polynomials are denoted by \( \overline{R}(\alpha; q) \). They showed that the rank sequences of these posets are symmetric.

\begin{theorem}[Theorem 1.6 of~\cite{OR23}]\label{thm:circular}
    The rank polynomials of circular fence posets are symmetric.
\end{theorem}

Furthermore,~\cite{OOR24} established that the rank polynomials of circular fence posets are unimodal, with the exception of cases where the poset has the form \( (1,a,1,a) \) or \( (a,1,a,1) \) for some \( a \in \mathbb{N} \). In these exceptional cases, the rank sequence takes the form:
\[
(1, 2, \ldots, a+1, a, a+1, a, \ldots, 2, 1).
\]

\begin{theorem}[Theorem 5.1 of~\cite{OOR24}]\label{thm:circuni}
    The rank polynomials of circular fence posets \( \overline{P}(\alpha) \) are unimodal, except when \( \alpha = (a,1,a,1) \) or \( \alpha = (1,a,1,a) \) for some positive integer \( a \).
\end{theorem}

The above theorem naturally implies that except when \( \alpha = (a,1,a,1) \) or \( \alpha = (1,a,1,a) \), the rank polynomials of circular fence posets are almost interlacing as well.

\subsection{Posets in Cluster Algebras from Surfaces}

In what follows, we construct a fence poset \( (P_{\gamma}, \prec) \) associated with a (possibly generalized) arc \( \gamma \) on a marked surface \( (S, M) \) equipped with a triangulation \( T \). Our terminology and exposition largely follow the framework developed in~\cite{banaian2024skein}.

First, suppose that \( \gamma \) is an arc with both endpoints tagged plain. Choose an orientation on \( \gamma \), and let \( 1, \ldots, d \) be the sequence of arcs in \( T \) crossed by \( \gamma \), listed in the order determined by this orientation. We define a poset structure on the index set \( [d]:=\{1,\ldots,d\} \) as follows. Each pair of consecutive crossings \( i \) and \( i+1 \) corresponds to a triangle \( \Delta_i \) through which \( \gamma \) passes. Let \( s_i \) be the shared endpoint of \( i \) and \( i+1 \), which is also a vertex of \( \Delta_i \). If \( s_i \) lies to the right of \( \gamma \) (with respect to its orientation), then set \( i \succ i+1 \); otherwise, set \( i \prec i+1 \). The resulting poset is called a \emph{fence poset}, as its Hasse diagram forms a path graph. This procedure applies regardless of whether \( \gamma \) has self-intersections.

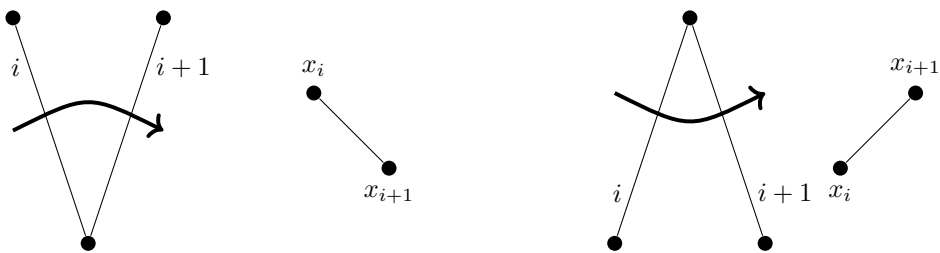
\begin{figure}[H]
\centering
\begin{tikzpicture}
    \coordinate (P1) at (0, 3);
    \coordinate (P2) at (1, 0);
    \coordinate (P3) at (2, 3);
    \coordinate (P4) at (4, 2);
    \coordinate (P5) at (5, 1);

    \node[circle, fill=black, inner sep=2pt] (P1) at (P1) {};
    \node[circle, fill=black, inner sep=2pt] (P2) at (P2) {};
    \node[circle, fill=black, inner sep=2pt] (P3) at (P3) {};
    \node[circle, fill=black, inner sep=2pt, label=above:$v_i$] (P4) at (P4) {};
    \node[circle, fill=black, inner sep=2pt, label=below:$v_{i+1}$] (P5) at (P5) {};
    
    \draw[-] (P1) -- (P2) node[pos=0.2, left] {$i$};
    \draw[-] (P2) -- (P3) node[pos=0.8, right] {$i+1$};
    \draw[-] (P4) -- (P5);

    \coordinate (Q1) at (0, 1.5);
    \coordinate (Q2) at (1, 2);
    \coordinate (Q3) at (2, 1.5);

    \draw[smooth, ->, line width=1.5pt] (Q1) .. controls (Q2) .. (Q3);

    \coordinate (P6) at (8, 0);
    \coordinate (P7) at (9, 3);
    \coordinate (P8) at (10, 0);
    \coordinate (P9) at (11, 1);
    \coordinate (P10) at (12, 2);

    \node[circle, fill=black, inner sep=2pt] (P6) at (P6) {};
    \node[circle, fill=black, inner sep=2pt] (P7) at (P7) {};
    \node[circle, fill=black, inner sep=2pt] (P8) at (P8) {};
    \node[circle, fill=black, inner sep=2pt, label=below:$v_i$] (P9) at (P9) {};
    \node[circle, fill=black, inner sep=2pt, label=above:$v_{i+1}$] (P10) at (P10) {};
    \draw[-] (P6) -- (P7) node[pos=0.2, left] {$i$};
    \draw[-] (P7) -- (P8) node[pos=0.8, right] {$i+1$};
    \draw[-] (P9) -- (P10);

    \coordinate (Q4) at (8, 2);
    \coordinate (Q5) at (9, 1.5);
    \coordinate (Q6) at (10, 2);

    \draw[smooth, ->, line width=1.5pt] (Q4) .. controls (Q5) .. (Q6);
\end{tikzpicture}
\caption{Building a fence poset for an arc.}
\label{fig:closed_curve_fence_poset}
\end{figure}

Next, suppose that \( \gamma^{(p)} \) is notched at its starting point \( s(\gamma) = p \). Begin by constructing the fence poset for the underlying plain arc. Let \( \Delta_0 \) be the first triangle traversed by \( \gamma^{(p)} \); necessarily, \( \Delta_0 \) is bordered by \( 1 \) and two spokes at \( p \). Denote the set of all spokes at \( p \) in \( T \) by \( \sigma_1, \ldots, \sigma_m \), listed counterclockwise, where \( \sigma_1 \) is the counterclockwise neighbor of \( 1 \). If an arc has both endpoints at \( p \), it receives two labels depending on their position in the cyclic order around \( p \). We introduce elements \( 1^s, \ldots, m^s \) into the poset and define the relations \( m^s \succ \dots \succ 1^s \), together with \( 1^s \prec 1 \) and \( m^s \succ 1 \). For arcs notched at their terminal point, we perform an analogous construction with elements \( 1^t, \ldots, m^t \). For arcs notched at both endpoints, we apply both constructions. The resulting posets are called \emph{loop fence posets}, corresponding to the loop graphs introduced by Wilson in~\cite{wilson2020surface}. An example of a loop fence poset is illustrated in Figure~\ref{fig:exnotched}.

\begin{figure}[htbp]
    \centering
    \begin{tikzpicture}[scale=1] 
        \draw (4.5,-1) to (5,0);
        \draw (4,1.5) to (5,0);
        \draw (6,0) to (5,0);
        \draw (3,-1.2) to (4.5,-1);
        \draw (4.5,-1) to (4,1.5);
        \draw (3,-1.2) to (4,1.5);
        \draw (3,-1.2) to (3,1.5);
        \draw (0,1) to (1.5,0);
        \draw (0,1) to (0,-1);
        \draw (0,-1) to (1.5,0);
        \draw (3,1.5) to (1.5,0);
        \draw (3,-1.2) to (1.5,0);
        \draw (3,-1.2) to (0,-1);
        \draw (3,1.5) to (0,1);
        \draw (3,1.5) to (4,1.5);
        \draw (4,1.5) to (6,0);
        \draw (4.5,-1) to (6,0);

        \draw[red,thick] (2.1,0.2) to (1.9,-0.2);
        \draw[red,thick] (1.9,0.2) to (2.1,-0.2);
        \draw[red,thick] (2.1,0.2) to (1.9,0.2);
        \draw[red,thick] (2.1,-0.2) to (1.9,-0.2);

        \node[fill=white,scale=0.7] at (0.6,0.5) {$\sigma_2$};
        \node[fill=white,scale=0.7] at (0.8,-0.4) {$\sigma_3$};
        \node[fill=white,scale=0.7] at (2,-0.5) {$\sigma_4$};
        \node[fill=white,scale=0.7] at (2,0.4) {$\sigma_1$};
        \node[fill=white,scale=0.7] at (3,0.5) {$\tau_1$};
        \node[fill=white,scale=0.7] at (3.5,0.2) {$\tau_2$};
        \node[fill=white,scale=0.7] at (4.15,0.3) {$\tau_3$};

        \draw[red,thick] (1.5,0) to (5,0);
        \node[red, fill=white, scale=0.8] at (2.8,0.0) {$\gamma_1$};

        \node[fill=white,scale=0.7] at (4.7,0.4) {$\eta_3$};
        \node[fill=white,scale=0.7] at (5.5,0) {$\eta_2$};
        \node[fill=white,scale=0.7] at (4.8,-0.3) {$\eta_1$};

        \node[scale=0.8] at (1.5,0.2) {$p$};
        \node[scale=0.8] at (5.1,0.2) {$q$};

        \draw[fill=black] (5,0) circle [radius=2pt];
        \draw[fill=black] (6,0) circle [radius=2pt];
        \draw[fill=black] (4,1.5) circle [radius=2pt];
        \draw[fill=black] (4.5,-1) circle [radius=2pt];
        \draw[fill=black] (3,-1.2) circle [radius=2pt];
        \draw[fill=black] (3,1.5) circle [radius=2pt];
        \draw[fill=black] (0,1) circle [radius=2pt];
        \draw[fill=black] (0,-1) circle [radius=2pt];
        \draw[fill=black] (1.5,0) circle [radius=2pt];
    \end{tikzpicture}
    \quad 
    \begin{tikzpicture}[baseline=(s1.base)] 
        \node (s4) at (0,0) {$\sigma_4$};
        \node (s3) at (0.5,-1) {$\sigma_3$};
        \node (s2) at (1,-2) {$\sigma_2$};
        \node (s1) at (1.5,-3) {$\sigma_1$};
        \node (t1) at (2,-2) {$\tau_1$};
        \node (t2) at (2.5,-3) {$\tau_2$};
        \node (t3) at (3,-2) {$\tau_3$};

        \draw (s4) to (t1);
        \draw (s4) to (s3);
        \draw (s3) to (s2);
        \draw (s2) to (s1);
        \draw (s1) to (t1);
        \draw (t1) to (t2);
        \draw (t2) to (t3);
    \end{tikzpicture}

    \caption{Example of notched tagged arc and its fence poset.}
    \label{fig:exnotched}
\end{figure}

Now consider the case where \( \gamma \) is a closed curve. Choose a point \( a \in \gamma \) that is not a point of intersection with \( T \), and fix an orientation on \( \gamma \). Treat \( \gamma \) as an arc with \( s(\gamma) = t(\gamma) = a \), and build the fence poset on \( [d] \) accordingly. Since \( \tau_{i_1} \) and \( \tau_{i_d} \) share an endpoint which is a vertex of the triangle containing \( a \), we close the cycle by setting \( d \succ 1 \) if this shared endpoint lies to the right of \( \gamma \); otherwise, set \( d \prec 1 \). These posets turn out to be circular fence posets.

\begin{theorem}[Theorems 5.4, 5.7 of~\cite{musiker2011positivity}; Theorem 7.9 of~\cite{wilson2020surface}]\label{Thm:PMPoset}
Let \( S \) be a surface with triangulation \( T \), and let \( \gamma \) be an arc on \( S \) such that \( \gamma^0 \notin T \).
\begin{enumerate}
    \item If \( \gamma \) is a plain arc, then there exists a bijection between \( J(F_\gamma) \) and the terms of the expansion of $x_{\gamma,T}$.
    \item If \( \gamma \) is a tagged arc or a closed curve, then there exists a bijection between \( J(F_\gamma) \) and the terms of the expansion of $x_{\gamma,T}$.
\end{enumerate}
\end{theorem}

Given an arc \( \tau \in T \), define
\[
x_{\mathrm{CCW}}(\tau) = x_{\tau_j} x_{\tau_k},
\]
where \( \tau_j, \tau_k \in T \) are the counterclockwise neighbors of \( \tau \) within the triangles adjacent to it. If one or both neighbors lie on the boundary, their contribution is omitted. Similarly, define \( x_{\mathrm{CW}}(\tau) \) using the clockwise neighbors. We then set
\[
\hat{y}_{\tau} = \frac{x_{\mathrm{CCW}}(\tau)}{x_{\mathrm{CW}}(\tau)} \, y_{\tau}.
\]

This is the standard definition of the \( \hat{y} \)-variables in cluster algebras from surfaces. In particular, if we specialize \( x_i = 1 \) for all \( i \), then \( \hat{y}_\tau = y_\tau \).

For an ideal \( I \in J(F_\gamma) \), define its weight by
\[
\operatorname{wt}(I) = \prod_{\tau \in I} y_\tau,
\]
interpreted as a multiset product, with the convention \( \operatorname{wt}(\emptyset) = 1 \). In cluster algebras from surfaces, the $F$-polynomial associated to an arc $\gamma$ with respect to an initial triangulation $T$ will be denoted
\[
F_{\gamma}^T(y_1,\ldots,y_n).
\]

\begin{prop}[Proposition~1 of~\cite{banaian2024skein}, Theorem~4.2 of~\cite{ouguz2024cluster}]\label{prop:IDEAL}
Let \( \gamma \) be an arc or closed curve on a marked surface \( (S, M) \) with triangulation \( T \). Upon specializing \( x_i = 1 \) for all \( i \), the $F$-polynomial $F^T_{\gamma}(y_1,\ldots,y_n)$ of the cluster variable \( x_\gamma \in \mathcal{A}(S, M) \), expressed in the seed associated with \( T \), is given by
\[
F^T_{\gamma}(y_1,\ldots,y_n) = \sum_{I \in J(F_\gamma)} \operatorname{wt}(I).
\]
\end{prop}

By combining Theorem~\ref{Thm:PMPoset} with Proposition~\ref{prop:IDEAL}, we obtain the following:

\begin{prop}
Identifying all \( y_i \) with \( q \), the $F$-polynomial of \( x_\gamma \) is equal to the rank polynomial \( R(F_\gamma; q) \).
\end{prop}

Consequently, Theorem~\ref{thm:rankuni} can be reformulated as follows:

\begin{theorem}\label{thm:rankuni2}
The rank polynomial \( R(P_\gamma; q) \) is almost interlacing when \( \gamma \) is a plain arc.
\end{theorem}

In the next section we prove that the rank polynomial of any loop fence poset is unimodal; more precisely, its rank sequence is almost interlacing. Furthermore, we conjecture that this property extends to a broader class of tagged arcs, including those with self-intersections, suggesting that the almost interlacing property is a universal characteristic of their associated rank sequences. It is worth noting that by Theorem~\ref{thm:circuni}, this property is already established for closed curves—with the exception of cases where the composition $\alpha$ takes the form $(1,a,1,a)$ or $(a,1,a,1)$ for positive integer $a$. In all other instances, the rank sequence is both unimodal and symmetric, which implies the almost interlacing property.

\section{Unimodality of Loop Fence Posets}\label{sec:uniloop}

A natural question is whether Theorem~\ref{thm:rankuni} extends to loop fence posets; that is, whether the rank polynomial of the fence poset associated to a notched tagged arc \( \gamma \) remains unimodal. In this section, we show that the rank polynomials of loop fence posets are indeed unimodal. To verify unimodality, we consider the following two conditions on a rank sequence $(r_0,\ldots,r_{n+1})$:

\begin{align*}
    r_{i} \leq r_{n-i} \quad \text{and}\quad
    r_{n+1-i} \leq r_{i+1} \quad \text{for } 0 \leq i \leq \left\lfloor \tfrac{n-1}{2} \right\rfloor \tag{Ineq A}\\
    r_{0} \leq r_{1} \leq \ldots \leq r_{\left\lfloor \frac{n}{2} \right\rfloor} \quad \text{and} \quad 
    r_{\left\lceil \frac{n}{2} \right\rceil + 1} \geq \ldots \geq r_{n+1} \tag{Ineq B} \label{eq:ineqB}
\end{align*}

Then, we have the following useful lemma.

\begin{lemma}\label{lem:ABau}
If a sequence $(r_0,r_1,\dots,r_{n+1})$ satisfies both \eqref{eq:ineqA} and \eqref{eq:ineqB}, then the sequence is almost interlacing.
\end{lemma}

\begin{proof}
It is enough to show the unimodality, since \eqref{eq:ineqA} is already given. First, we consider the case where $n$ is even. Since \eqref{eq:ineqB} implies unimodality, we are done. Suppose that $n$ is odd. Then, we obtain
\begin{align*}
    r_{0} \leq r_{1} \leq \ldots \leq r_{\frac{n-1}{2}} \quad \text{and} \quad 
    r_{\frac{n+3}{2}} \geq \ldots \geq r_{n+1}.
\end{align*}
By \eqref{eq:ineqA}, we have $r_{\frac{n-1}{2}}\le r_{\frac{n+1}{2}}$. Hence, $(r_0,r_1,\dots,r_{n+1})$ is unimodal.
\end{proof}

To model punctured surfaces, we incorporate an additional order relation into the fence poset structure.

\begin{definition}[Augmented and loop Fence Posets]
    Let $\alpha = (\alpha_1, \alpha_2, \ldots, \alpha_s)$ be a composition of $n$. The \emph{$(i,j)$-fence poset}, $P^{(i,j)}(\alpha)$, is constructed by extending the standard fence poset $P(\alpha)$ on the vertex set $\{v_1, v_2, \ldots, v_{n+1}\}$ with the additional covering relation $v_i \prec v_j$.

    A poset $P^{(i,j)}(\alpha)$ is termed a \emph{singly notched loop fence poset} if the pair $(i,j)$ satisfies one of the following boundary conditions:
    \begin{itemize}
        \item \textbf{Type I:} The notch occurs at the beginning of the fence:
        \[
        (i,j) = \begin{cases} 
            (1, \alpha_1 + 2) & \text{if } \alpha_1 > 0, \\ 
            (\alpha_2 + 2, 1) & \text{if } \alpha_1 = 0. 
        \end{cases}
        \]
        This poset is denoted by $P^{(1)}(\alpha)$.

        \item \textbf{Type II:} The notch occurs at the end of the fence:
        \[
        (i,j) = \begin{cases} 
            (n+1, n - \alpha_s) & \text{if } s \text{ is even}, \\ 
            (n - \alpha_s, n+1) & \text{if } s \text{ is odd}. 
        \end{cases}
        \]
        This poset is denoted by $P^{(n+1)}(\alpha)$.
    \end{itemize}

    If the poset satisfies both Type I and Type II conditions simultaneously, it is designated as a \emph{doubly notched loop fence poset}, denoted $P^{(1)(n+1)}(\alpha)$. Both singly notched loop fence posets and doubly loop notched fence posets are referred to as \emph{notched fence posets}.
\end{definition}

For brevity, we shall refer to \emph{singly notched loop fence posets} simply as \emph{singly notched posets}. Analogously, \emph{doubly notched loop fence posets} will be referred to as \emph{doubly notched posets}.

\begin{figure}[H]
    \centering
\begin{tikzpicture}
    \draw (0,0.5) to (0.5,1);
    \draw (0.5,1) to (2,-0.5);
    \draw (2,-0.5) to (3.5,1);
    \draw (1.7,-0.2) to (3.5,1);
    \node at (4.1,1) {$v_{n+1}$};
    \node at (2,-0.8) {$v_{n-\alpha_{s}+1}$};
    \node at (1,-0.3) {$v_{n-\alpha_{s}}$};
    \filldraw[black] (3.5,1) circle (2pt);
    \filldraw[black] (1.7,-0.2) circle (2pt);
    \filldraw[black] (2,-0.5) circle (2pt);

    \draw (6,0) to (6.5,-0.5);
    \draw (6.5,-0.5) to (8,1);
    \draw (8,1) to (9.5,-0.5);
    \draw (7.7,0.7) to (9.5,-0.5);
    \node at (10.1,-0.5) {$v_{n+1}$};
    \node at (8,1.3) {$v_{n-\alpha_{s}+1}$};
    \node at (7,0.8) {$v_{n-\alpha_{s}}$};
    \filldraw[black] (9.5,-0.5) circle (2pt);
    \filldraw[black] (7.7,0.7) circle (2pt);
    \filldraw[black] (8,1) circle (2pt);

\end{tikzpicture}
    \caption{Structure of \( P^{(n+1)}(\alpha) \) when \( s \) is odd(left) and when \(s\) is even(right)}
    \label{fig:RPFP}
\end{figure}
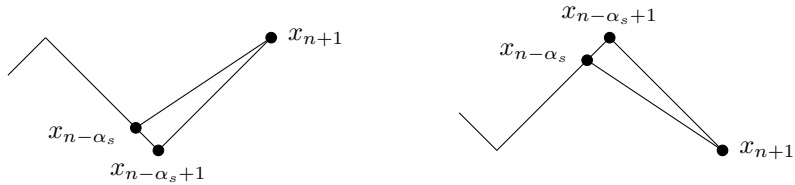

Note that this construction coincides with the loop fence posets associated to notched arcs. A singly notched poset corresponds to a singly notched arc, while a doubly notched poset corresponds to a doubly notched arc. We denote the associated rank polynomial by \( R^{(i,j)}(\alpha; q) \) for a general \((i,j)\)-fence poset, \( R^{(1)}(\alpha; q) \) for \( P^{(1)}(\alpha) \), \( R^{(n+1)}(\alpha; q) \) for \( P^{(n+1)}(\alpha) \), and \( R^{(1)(n+1)}(\alpha; q) \) for \( P^{(1)(n+1)}(\alpha) \).

\begin{remark}
    The structure of a loop fence poset can be viewed as the augmentation of a fence poset with a circular fence poset. This construction suggests that while the rank polynomial is modified, the underlying combinatorial properties—specifically the interlacing property—remain largely invariant. Consequently, it is expected that the property of unimodality is preserved in this setting.
\end{remark}

\subsection{Unimodality of Singly Notched Posets}

In this subsection, we establish the unimodality of rank polynomials for singly notched posets. Specifically, for any composition \( \alpha \), the poset \( P^{(1)}(\alpha) \) shares the same rank polynomial as the reflected poset \( P^{(n+1)}(\alpha) \). This equivalence follows from the fact that reflecting a poset preserves its partial order structure, and hence its rank polynomial remains unchanged. Therefore, it suffices to prove the unimodality of \( P^{(n+1)}(\alpha) \), which in turn implies the same for \( P^{(1)}(\alpha) \). As a result, we conclude that the rank polynomials of all singly notched posets are unimodal.

\begin{prop}\label{prop:almint}
Let \( r(\alpha) = (r_0, r_1, \ldots, r_{n}) \) be an almost interlacing sequence. Then:
\begin{itemize}
    \item If \( i \leq j \) and \( i + j \leq n - 2 \), then \( r_i \leq r_j \).
    \item If \( i \leq j \) and \( i + j \geq n + 2 \), then \( r_i \geq r_j \).
\end{itemize}
\end{prop}

\begin{proof}
When \( j \leq \left\lfloor \frac{n}{2} \right\rfloor \), observe that unimodality gives \( r_i \leq r_j \). For the case where \( j > \left\lfloor \frac{n}{2} \right\rfloor \), observe that \eqref{eq:ineqA} gives \( r_i \leq r_{n-1  - i} \), and unimodality and \( i + j \leq n - 2 \) implies \( r_{n  -1- i} \leq r_j \). Combining these yields \( r_i \leq r_j \). A similar argument applies to the second inequality.
\end{proof}

\begin{lemma}\label{lem:left}
    Let $\overline{P}(\alpha)$ be a circular fence poset and $r(\alpha)=(r_0,r_1,\dots,r_n)$. Then, if $i\le j$ and $i+j\le n-2$, then $r_i\le r_j$. Moreover, if $i\le j$ and $i+j\ge n+2$, then $r_i\ge r_j$.
\end{lemma}
\begin{proof}
If $\alpha\neq(1,a,1,a)$, then $r(\alpha)$ is almost unimodal since it is symmetric and unimodal, and the statement from Proposition~\ref{prop:almint}. Now, we consider the case where $\alpha=(1,a,1,a)$. Then, $n=2a+2$ and $r(\alpha)=(1,2,\dots,a,a+1,a,a+1,a,\dots,1)$. When $j\neq a+1$, we can see that $a_i\le a_j$ by the same method above. Assume that $j=a+1$. It follows from $r(\alpha)=(1,2,\dots,a,a+1,a,a+1,a,\dots,1)$. The second case can be proved in the same way by applying reflection on the indices by changing $i$ to $n-i$ for $0\leq i \leq n$.
\end{proof}

\begin{lemma}\label{lem:eq1}
    The rank polynomial $R^{(n+1)}(\alpha; q)$ of the singly notched poset $P^{(n+1)}(\alpha)$ can be decomposed into the difference of two circular fence poset polynomials, weighted by a power of $q$, $q^{\alpha_s+1}$.
\end{lemma}
\begin{proof}
Let us consider the case when \( s \) is even. The case for odd \( s \) follows by simply exchanging each occurrence of \( v_n \) with \( v_{n - \alpha_s + 1} \).

Define \( P_T^{(n+1)}(\alpha) \) as the poset obtained by adding a new vertex \( v_T \) lying above both \( v_1 \) and \( v_{n - \alpha_s + 1} \) in the fence poset associated to \( \alpha \). Let \( R_T^{(n+1)}(\alpha; q) \) be the corresponding rank polynomial. The ideals of \( P_T^{(n+1)}(\alpha) \) are partitioned into those that include \( v_T \) and those that do not.

Let \( \beta \) be the composition obtained from \( \alpha \) by removing \( v_T \) and all vertices below it. Let \( d_1 \) denote the number of vertices below \( v_T \). Note that \( d_1 \geq \alpha_{s-1} + \alpha_s + 2 \), since \( v_T \) lies above every vertex \( v_i \) with \( n - \alpha_{s-1} - \alpha_s + 1 \leq i \leq n+1 \), and also above \( v_1 \). Then we have:
\[
R_T^{(n+1)}(\alpha; q) = R^{(n+1)}(\alpha; q) + q^{d_1 + 1} R(\beta; q).
\]

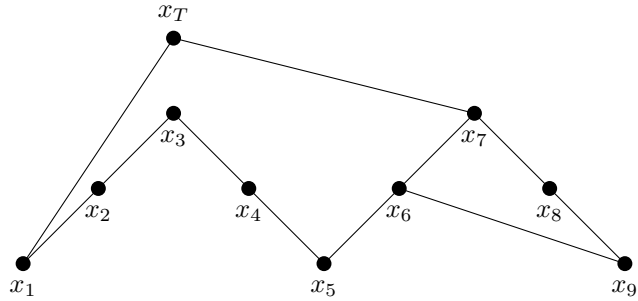
\begin{figure}[H]
    \centering
    \begin{tikzpicture}
        \coordinate (x1) at (1, 2);
        \coordinate (x2) at (2, 3);
        \coordinate (x3) at (3, 4);
        \coordinate (x4) at (4, 3);
        \coordinate (x5) at (5, 2);
        \coordinate (x6) at (6, 3);
        \coordinate (x7) at (7, 4);
        \coordinate (x8) at (8, 3);
        \coordinate (x9) at (9, 2);
        \coordinate (xT) at (3, 5);

        \node[circle, fill=black, inner sep=2pt, label=below:$v_{1}$] at (x1) {};
        \node[circle, fill=black, inner sep=2pt, label=below:$v_{2}$] at (x2) {};
        \node[circle, fill=black, inner sep=2pt, label=below:$v_{3}$] at (x3) {};
        \node[circle, fill=black, inner sep=2pt, label=below:$v_{4}$] at (x4) {};
        \node[circle, fill=black, inner sep=2pt, label=below:$v_{5}$] at (x5) {};
        \node[circle, fill=black, inner sep=2pt, label=below:$v_{6}$] at (x6) {};
        \node[circle, fill=black, inner sep=2pt, label=below:$v_{7}$] at (x7) {};
        \node[circle, fill=black, inner sep=2pt, label=below:$v_{8}$] at (x8) {};
        \node[circle, fill=black, inner sep=2pt, label=below:$v_{9}$] at (x9) {};
        \node[circle, fill=black, inner sep=2pt, label=above:$v_{T}$] at (xT) {};

        \draw (x1) -- (x3);
        \draw (x3) -- (x5);
        \draw (x5) -- (x7);
        \draw (x7) -- (x9);
        \draw (x6) -- (x9);
        \draw (x1) -- (xT);
        \draw (x7) -- (xT);
    \end{tikzpicture}
    \caption{An example of \( P_T^{(n+1)}(\alpha) \) when \( \alpha = (2,2,2,2) \)}
    \label{fig:PT}
\end{figure}

Next, let \( \overline{P_T}(\alpha) \) be the circular poset obtained by removing the relation \( v_{n - \alpha_s} \prec v_{n - \alpha_s + 1} \) from \( P_T^{(n+1)}(\alpha) \), and let \( \overline{R_T}(\alpha; q) \) denote its rank polynomial.

Let \( \gamma \) be the poset obtained from \( \overline{P_T}(\alpha) \) by removing \( v_{n-\alpha_s} \) and all the vertices above it and removing \( v_{n - \alpha_s + 1} \) along with all the vertices below it. The ideals of \( \overline{P_T}(\alpha) \) then consist of:
(1) those in \( P_T^{(n+1)}(\alpha) \) and
(2) additional ideals of \( \gamma \), which correspond to including \( v_{n - \alpha_s + 1} \) but not \( v_{n - \alpha_s} \), weighted by \( q^{\alpha_s + 1} \).

Thus, we have the following.
\[
\overline{R_T}(\alpha; q) = R_T^{(n+1)}(\alpha; q) + q^{\alpha_s + 1} R(\gamma; q),
\]
which implies
\[
R^{(n+1)}(\alpha; q) = \overline{R_T}(\alpha; q) - q^{d_1 + 1} R(\beta; q) - q^{\alpha_s + 1} R(\gamma; q).
\]

Now, define the poset \( \delta \) as follows. Let \( \beta \cup \gamma \) be the union of the underlying sets of \( \beta \) and \( \gamma \), preserving their partial orders. \( v_T \) will be one endpoint of $\beta\cup \gamma$ and let the other vertex as \( v_P \) where \( 1 \leq P \leq n - \alpha_s - 1 \). To extend this structure, we introduce an additional vertex $v_{P+1}$ and prescribe a relation between $v_P$ and $v_{P+1}$ consistent with the composition $\alpha$. Finally, we impose the additional relation $v_T \prec v_{P+1}$ and relabel $v_{P+1}$ as $v_{\tilde{T}}$. The resulting structure constitutes the circular fence poset $\delta$.

\begin{table}[htbp]
    \centering
    \begin{tabular}{|c|c|c|}
    \hline
           \begin{tikzpicture}
        \coordinate (x2) at (2, .5);
        \coordinate (x3) at (3, 1);
        \coordinate (x4) at (4, .5);

        \node[circle, fill=black, inner sep=2pt, label={[below,xshift=5pt,yshift=-5pt]{$v_{2}$}}] (nodex2) at (x2) {};
        \node[circle, fill=black, inner sep=2pt, label=below:$v_{3}$] (nodex3) at (x3) {};
        \node[circle, fill=black, inner sep=2pt, label=below:$v_{4}$] (nodex4) at (x4) {};

        \draw[-] (x2) -- (x3);
        \draw[-] (x3) -- (x4);
    \end{tikzpicture}
  & 
      \begin{tikzpicture}
        \coordinate (x1) at (1, 0);
        \coordinate (x2) at (2, .5);
        \coordinate (x3) at (3, 1);
        \coordinate (x4) at (4, .5);
        \coordinate (x5) at (5, 0);
        \coordinate (xT) at (3, 2);

        \node[circle, fill=black, inner sep=2pt, label=below:$v_{1}$] (nodex1) at (x1) {};
        \node[circle, fill=black, inner sep=2pt, label={[below,xshift=5pt,yshift=-5pt]{$v_{2}$}}] (nodex2) at (x2) {};
        \node[circle, fill=black, inner sep=2pt, label=below:$v_{3}$] (nodex3) at (x3) {};
        \node[circle, fill=black, inner sep=2pt, label=below:$v_{4}$] (nodex4) at (x4) {};
        \node[circle, fill=black, inner sep=2pt, label={[above,xshift=4pt]{$v_{5}$}}] (nodex5) at (x5) {};
        \node[circle, fill=black, inner sep=2pt, label={[above,xshift=4pt]{$v_{T}$}}] (nodexT) at (xT) {};

        \draw[-] (x1) -- (x3);
        \draw[-] (x3) -- (x5);
        \draw[-] (x1) -- (xT);
    \end{tikzpicture}

&    \begin{tikzpicture}
        \coordinate (x1) at (1, 0);
        \coordinate (x2) at (2, .5);
        \coordinate (x3) at (3, 1);
        \coordinate (x4) at (4, .5);
        \coordinate (x5) at (5, 0);
        \coordinate (xT) at (3, 2);

        \node[circle, fill=black, inner sep=2pt, label=below:$v_{1}$] (nodex1) at (x1) {};
        \node[circle, fill=black, inner sep=2pt, label={[below,xshift=5pt,yshift=-5pt]{$v_{2}$}}] (nodex2) at (x2) {};
        \node[circle, fill=black, inner sep=2pt, label=below:$v_{3}$] (nodex3) at (x3) {};
        \node[circle, fill=black, inner sep=2pt, label=below:$v_{4}$] (nodex4) at (x4) {};
        \node[circle, fill=black, inner sep=2pt, label={[above,xshift=4pt]{$v_{5}=v_P$}}] (nodex5) at (x5) {};
        \node[circle, fill=black, inner sep=2pt, label={[above,xshift=4pt]{$v_{T}$}}] (nodexT) at (xT) {};

        \draw[-] (x1) -- (x3);
        \draw[-] (x3) -- (x5);
        \draw[-] (x1) -- (xT);
    \end{tikzpicture}\\
    \hline
    $\beta$& $\gamma$ & $\beta\cup\gamma$\\
    \hline
    \end{tabular}
    \caption{$\beta$, $\gamma$, and $\beta\cup\gamma$ for $P^{(n+1)}_T(\alpha)$ when $\alpha=(2,2,2,2)$}
    \label{tab:topex}
\end{table}

\begin{figure}[H]
    \centering
    \begin{tikzpicture}
        \coordinate (x1) at (1, 0);
        \coordinate (x2) at (2, .5);
        \coordinate (x3) at (3, 1);
        \coordinate (x4) at (4, .5);
        \coordinate (x5) at (5, 0);
        \coordinate (xT) at (3, 2);
        \coordinate (xT') at (6, 3);

        \node[circle, fill=black, inner sep=2pt, label=below:$v_{1}$] (nodex1) at (x1) {};
        \node[circle, fill=black, inner sep=2pt, label={[below,xshift=5pt,yshift=-5pt]{$v_{2}$}}] (nodex2) at (x2) {};
        \node[circle, fill=black, inner sep=2pt, label=below:$v_{3}$] (nodex3) at (x3) {};
        \node[circle, fill=black, inner sep=2pt, label=below:$v_{4}$] (nodex4) at (x4) {};
        \node[circle, fill=black, inner sep=2pt, label={[below,yshift=-3pt]{$v_{5}=v_P$}}] (nodex5) at (x5) {};
        \node[circle, fill=black, inner sep=2pt, label={[below,xshift=4pt,yshift=-5pt]{$v_{T}$}}] (nodexT) at (xT) {};
        \node[circle, fill=black, inner sep=2pt, label={[above,xshift=4pt]{$v_6=v_{\tilde{T}}$}}] (nodexT') at (xT') {};

        \draw[-] (x1) -- (x3);
        \draw[-] (x3) -- (x5);
        \draw[-] (x5) -- (xT');
        \draw[-] (x1) -- (xT);
        \draw[-] (xT) -- (xT');
    \end{tikzpicture}
    \caption{$\delta$ for $P^{(n+1)}_T(\alpha)$ when $\alpha=(2,2,2,2)$}
    \label{fig:delta}
\end{figure}
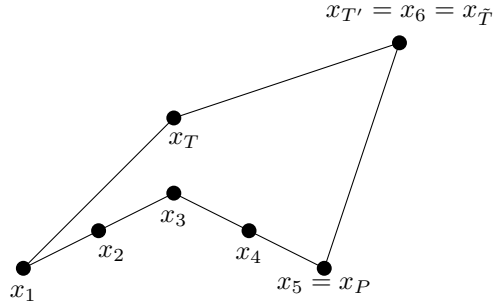

There are now two types of ideals in \( \delta \):
- those that include \( v_{\tilde{T}} \), contributing \( q^{d_1 - \alpha_s} R(\beta; q) \), - and those that exclude it, contributing \( R(\gamma; q) \).

Thus, we obtain:
\[
\overline{R}(\delta; q) = q^{d_1 - \alpha_s} R(\beta; q) + R(\gamma; q).
\]

Substituting into the earlier expression gives:
\begin{align}\label{eq:1}
R^{(n+1)}(\alpha; q) = \overline{R_T}(\alpha; q) - q^{\alpha_s+ 1} \overline{R}(\delta; q).
\end{align}
\end{proof}

\begin{lemma}\label{lem:eq2}
    For the rank polynomial $R^{(n+1)}(\alpha; q)$ associated with a singly notched poset $P^{(n+1)}(\alpha)$, $qR^{(n+1)}(\alpha; q)$ can be expressed as the difference between two circular fence poset polynomials.
\end{lemma}

\begin{proof}
Similarly, let \( P_B^{(n+1)}(\alpha) \) be the poset obtained by adding a new vertex \( v_B \) lying below both \( v_1 \) and \( v_{n+1} \) in the poset \( P^{(n+1)}(\alpha) \). Let \( R_B^{(n+1)}(\alpha; q) \) denote its rank polynomial. 

\begin{figure}[H]
    \centering
    \begin{tikzpicture}
        \coordinate (x1) at (1, 0);
        \coordinate (x2) at (2, 1);
        \coordinate (x3) at (3, 2);
        \coordinate (x4) at (4, 1);
        \coordinate (x5) at (5, 0);
        \coordinate (x6) at (6, 1);
        \coordinate (x7) at (7, 2);
        \coordinate (x8) at (8, 1);
        \coordinate (x9) at (9, 0);
        \coordinate (xB) at (5, -1);

        \node[circle, fill=black, inner sep=2pt, label=below:$v_{1}$] (nodex1) at (x1) {};
        \node[circle, fill=black, inner sep=2pt, label=below:$v_{2}$] (nodex2) at (x2) {};
        \node[circle, fill=black, inner sep=2pt, label=below:$v_{3}$] (nodex3) at (x3) {};
        \node[circle, fill=black, inner sep=2pt, label=below:$v_{4}$] (nodex4) at (x4) {};
        \node[circle, fill=black, inner sep=2pt, label=below:$v_{5}$] (nodex5) at (x5) {};
        \node[circle, fill=black, inner sep=2pt, label=below:$v_{6}$] (nodex6) at (x6) {};
        \node[circle, fill=black, inner sep=2pt, label=below:$v_{7}$] (nodex7) at (x7) {};
        \node[circle, fill=black, inner sep=2pt, label=below:$v_{8}$] (nodex8) at (x8) {};
        \node[circle, fill=black, inner sep=2pt, label=below:$v_{9}$] (nodex9) at (x9) {};
        \node[circle, fill=black, inner sep=2pt, label=below:$v_{B}$] (nodexT) at (xB) {};

        \draw[-] (x1) -- (x3);
        \draw[-] (x3) -- (x5);
        \draw[-] (x5) -- (x7);
        \draw[-] (x7) -- (x9);
        \draw[-] (x6) -- (x9);
        \draw[-] (x1) -- (xB);
        \draw[-] (x9) -- (xB);
    \end{tikzpicture}
    \caption{An example of \( P_B^{(n+1)}(\alpha) \) when \( \alpha = (2,2,2,2) \)}
    \label{fig:PB}
\end{figure}

Now, define \( \overline{P_B}(\alpha) \) as the poset obtained from \( P_B^{(n+1)}(\alpha) \) by removing the relation \( v_{n+1} \prec v_{n - \alpha_s} \), and let \( \overline{R_B}(\alpha; q) \) be the corresponding rank polynomial. 

The number of vertices below \( v_{n - \alpha_s} \) in \( \overline{P_B}(\alpha) \) is \(\alpha_{s-1}-1 \). By applying the same decomposition argument as before, we obtain:
\begin{align*}
    R_B^{(n+1)}(\alpha; q) &= q R^{(n+1)}(\alpha; q) + R(\beta'; q), \\
    \overline{R_B}(\alpha; q) &= R_B^{(n+1)}(\alpha; q) + q^{\alpha_{s-1}} R(\gamma'; q),
\end{align*}
where \( \beta' \) is the composition obtained from \( \overline{P_B}(\alpha) \) by removing \( v_B \) and all elements above it, and \( \gamma' \) is the composition obtained from \( \overline{P_B}(\alpha) \) by removing: \( v_{n - \alpha_s} \) and all vertices below it, and \( v_{n+1} \) and all vertices above it.

Combining the above, we have:
\[
q R^{(n+1)}(\alpha; q) = \overline{R_B}(\alpha; q) - R(\beta'; q) - q^{\alpha_{s-1}} R(\gamma'; q).
\]

Now, we construct a circular fence poset \( \delta' \) as follows. Form the union \( \beta' \cup \gamma' \), preserving the internal partial orders of \( \beta' \) and \( \gamma' \). Let \( v_B \) be the terminal vertex of \( \beta' \), and let \( v_Q \), with \( 1 \leq Q \leq n - \alpha_s - 1 \), be a terminal vertex of \( \gamma' \). Introduce a new vertex \( v_{Q+1} \), and insert the same relation between \( v_Q \) and \( v_{Q+1} \) as it appears in the original fence poset \( P(\alpha) \). Then add a relation identify \( v_{B} \succ v_{Q+1} \) and relabel $v_{Q+1}$ as \( v_{\tilde{B}} \), resulting in the circular fence poset \( \delta' \).

As before, we classify ideals of \( \delta' \) into those that contain \( v_{\tilde{B}} \) and those that do not. This yields:
\[
\overline{R}(\delta'; q) = R(\beta'; q) + q^{\alpha_{s-1}} R(\gamma'; q).
\]
Substituting into the earlier identity, we obtain:
\begin{align}
    q R^{(n+1)}(\alpha; q) = \overline{R_B}(\alpha; q) - \overline{R}(\delta'; q).
\end{align}

\end{proof}

\begin{lemma}\label{lem:sineqA}
    The rank sequence $r(\alpha)=(r_0,r_1,\ldots,r_{n+1})$ of a singly notched poset $P^{(n+1)}(\alpha)$ satisfies \eqref{eq:ineqA}.
\end{lemma}
\begin{proof}

By Lemma~\ref{lem:eq1}, $R^{(n+1)}(\alpha; q)$ can be expressed as the difference between two circular fence poset polynomials: \(R^{(n+1)}(\alpha; q) = \overline{R_T}(\alpha; q) - q^{\alpha_s+ 1} \overline{R}(\delta; q)\)

According to Theorem~\ref{thm:circular}, the polynomial $\overline{R_T}(\alpha; q)$ is symmetric. Note that $\deg(\overline{R_T}(\alpha; q)) = n + 2$ and $\deg(\overline{R}(\delta; q)) = n - \alpha_s + 1$. Let $(a_i)$ and $(b_i)$ denote the rank sequences of these respective polynomials.

The symmetry of $\overline{R_T}(\alpha; q)$ implies that its coefficients satisfy:\[
a_i = a_{n + 2 - i} \quad \text{for all } i.
\]
Since \( b_{n - \alpha_s + 1 - i} \geq 0 \), we get:
\[
r_i = a_i = a_{n + 2 - i} \geq a_{n + 2 - i} - b_{n - \alpha_s + 1 - i} = r_{n + 2 - i}, \quad \text{for } i \leq \alpha_s.
\]

In addition, for \( \alpha_s \leq i \leq \left\lfloor \frac{n+2}{2} \right\rfloor \), we have:
\[
r_i = a_i - b_{i - \alpha_s - 1} \geq a_{n + 2 - i} - b_{n - \alpha_s + 1 - i} = r_{n + 2 - i},
\]
because \( (i - \alpha_s - 1) + (n - \alpha_s + 1 - i) = n - 2\alpha_s \leq \deg(\overline{R}(\delta; q)) - 2 \), and thence by Lemma~\ref{lem:left}.

Therefore, we conclude that:
\[
r_i \geq r_{n + 2 - i} \quad \text{for all } i = 0,1, 2, \ldots, \left\lfloor \frac{n+2}{2} \right\rfloor.
\]

Let's move on to the other inequalities. By Lemma~\ref{lem:eq2}, $qR^{(n+1)}(\alpha; q)$ can be expressed as the difference between two circular fence poset polynomials: \(q R^{(n+1)}(\alpha; q) = \overline{R_B}(\alpha; q) - \overline{R}(\delta'; q).\)

Now, observe that \( \deg(\overline{R_B}(\alpha; q)) = n + 2 \), and \( \deg(\overline{R}(\delta'; q)) = n - \alpha_s + 1 \). Let \( (t_0, t_1, \dots, t_{n+2}) \) and \( (u_0, u_1, \dots, u_{n - \alpha_s + 1}) \) be the rank sequences of \( \overline{R_B}(\alpha; q) \) and \( \overline{R}(\delta'; q) \), respectively.

By Theorem~\ref{thm:circular} and Lemma~\ref{lem:left}, we have \( t_i = t_{n + 2 - i} \) and $u_{i+1}\ge u_{n+1-i}$ for \( i = 0, 1, \dots, \left\lfloor \frac{n - 1}{2} \right\rfloor \) since $i+1\le n+1-i$ and $(i+1)+(n+1-i)=n+2\ge (n-\alpha_s+1)+2$. Therefore, we obtain:
\begin{align}
    r_i = t_{i+1} - u_{i+1} \leq t_{n + 1 - i} - u_{n + 1 - i} = r_{n - i},
\end{align}
for \( i = 0, 1, \dots, \left\lfloor \frac{n - 1}{2} \right\rfloor \).

\end{proof}

\begin{remark}
Note that the proof follows the similar technique in the proof of Claim 5 in~\cite{OR23}. The difference comes from the usage of Lemma~\ref{lem:left} where we include the special case when the circular fence poset is $(1,a,1,a)$ or $(a,1,a,1)$ for some positive integer $a$ when its rank sequence is not unimodal.
\end{remark}

\begin{lemma}\label{lem:sineqB}
    The rank sequence \( r(\alpha) = (r_{0}, r_{1}, \ldots, r_{n+1}) \) of the singly notched poset \( P^{(n+1)}(\alpha) \) satisfies inequality~\textnormal{\eqref{eq:ineqB}}.
\end{lemma}

\begin{proof}
Let \( \alpha = (\alpha_1, \alpha_2, \ldots, \alpha_s) \) be a composition of \( n \). Let \( r(\alpha) = (r_0, \ldots, r_{n+1}) \) denote the rank sequence of the rank polynomial \( R^{(n+1)}(\alpha; q) \).

We first treat the case when \( s \) is odd. The ideals of the singly notched poset \( P^{(n+1)}(\alpha) \) can be partitioned into those that contain \( v_{n+1} \) and those that do not. Let \( \beta_1 \) be the composition obtained by removing \( v_{n+1} \) along with all elements below it, and let \( \beta_2 \) be obtained by removing only \( v_{n+1} \). Then we have the decomposition
\[
R^{(n+1)}(\alpha; q) = q^{d_1 + 1} R(\beta_1; q) + R(\beta_2; q),
\]
where \( d_1 \) is the number of elements lying below \( v_{n+1} \).
Let \( b_i \) and \( b_i' \) be the coefficients of \( R(\beta_1; q) \) and \( R(\beta_2; q) \), respectively. Then, we have
\[
R(\beta_1; q) = \sum_{i=0}^{n - d_1} b_i q^i, \quad R(\beta_2; q) = \sum_{i=0}^{n} b_i' q^i.
\]

By Theorem~\ref{thm:rankuni}, we have the following inequalities:
\begin{align*}
    &b_0 \le b_1 \le \dots \le b_{\lfloor (n - d_1)/2 \rfloor}, \\
&b_0' \le b_1' \le \dots \le b_{\lfloor n/2 \rfloor}'.
\end{align*}

Since \( q^{d_1 + 1} R(\beta_1; q) \) has increasing coefficients up to degree \( n/2 \) (including leading $d_{1} + 1$ zeros, actually up to \(\lfloor \frac{n - d_1}{2} \rfloor + d_1 + 1\)), the coefficients of their sum increase up to \( \lfloor n/2 \rfloor \), so
\[
r_0 \le r_1 \le \dots \le r_{\lfloor n/2 \rfloor}.
\]

Now consider the dual inequality. Again partition the ideals of \( P^{(n+1)}(\alpha) \) into those that contain the element \( v_{n - \alpha_s + 1} \) and those that do not. Let \( \gamma_1 \) be the composition obtained by removing \( v_{n - \alpha_s + 1} \) along with all elements above it, and let \( \gamma_2 \) be obtained by removing only \( v_{n - \alpha_s + 1} \). Then we have

\[
R^{(n+1)}(\alpha; q) = R(\gamma_1; q) + q R(\gamma_2; q).
\]

Let \( c_i \) and \( c_i' \) be the coefficients of \( R(\gamma_1; q) \) and \( R(\gamma_2; q) \), respectively. Then, we have
\[
R(\gamma_1; q) = \sum_{i=0}^{n - d_2} c_i q^i, \quad R(\gamma_2; q) = \sum_{i=0}^{n} c_i' q^i,
\]
where \( d_2 \) is the number of elements above \( v_{n - \alpha_s + 1} \). Applying Theorem~\ref{thm:rankuni} now yields
\begin{align*}
    &c_{\lceil (n - d_2)/2 \rceil} \ge c_{\lceil (n - d_2)/2 \rceil + 1} \ge \dots \ge c_{n - d_2}, \\
&c'_{\lceil n/2 \rceil} \ge c'_{\lceil n/2 \rceil + 1} \ge \dots \ge c'_n.
\end{align*}

Since \( R(\gamma_1; q) \) has decreasing coefficients from degree \( \lceil n/2 \rceil + 1 \) (including trailing $d_{2} + 1$ zeros), the coefficients of their sum decrease from \( \lceil n/2 \rceil + 1 \), so
\[
r_{\lceil n/2 \rceil + 1} \ge r_{\lceil n/2 \rceil + 2} \ge \dots \ge r_{n+1},
\]
which establishes the decreasing part of the sequence.

For even \( s \), the argument is entirely analogous. For the first inequality involving \( v_{n+1} \), replace \( v_{n+1} \) with \( v_{n - \alpha_s + 1} \); for the second inequality, reverse this replacement. In both cases, the rank sequence satisfies inequality~\textnormal{\eqref{eq:ineqB}}.
\end{proof}

\begin{theorem}\label{thm:RPFP}
    The rank polynomials of singly notched posets are almost interlacing.
\end{theorem}
\begin{proof}
Lemma~\ref{lem:sineqA}, Lemma~\ref{lem:sineqB}, and Lemma~\ref{lem:ABau} implies the statement.
\end{proof}

The following corollary will be used in the next section.

\begin{cor}\label{cor:eqauni}
    Let $r(\alpha)=(r_0,r_1,\ldots,r_{n+1})$ be a rank sequence of a singly notched poset $P^{(n+1)}(\alpha)$. Then, if $i\le j$ and $i+j\le n-2$, then $r_i\le r_j$. Moreover, if $i\le j$ and $i+j\ge n+2$, then $r_i\ge r_j$.
\end{cor}
\begin{proof}
    Directly follows from Proposition~\ref{prop:almint} and Theorem~\ref{thm:RPFP}.
\end{proof}

\subsection{Unimodality of Doubly Notched Posets}

The idea is very similar to the case of singly notched loop fence posets. Let the rank sequence of $P^{(1)(n+1)}(\alpha)$ as $(r_0,r_1,\ldots,r_{n+1})$. We will show both \eqref{eq:ineqA} and \eqref{eq:ineqB} hold for doubly notched case. The proof technique is almost identical to the proof of lemmas in the previous subsection.

\begin{lemma}\label{lem:doublyeqa}
    The rank sequence \( r(\alpha) = (r_{0}, r_{1}, \ldots, r_{n+1}) \) of the doubly notched poset $P^{(1)(n+1)}(\alpha)$ satisfies \eqref{eq:ineqA}.
\end{lemma}
\begin{proof}
    Let us consider when $\alpha_1\neq 0$ and $s$ is odd.

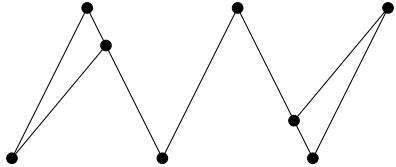
\begin{figure}[H]
    \centering
        \begin{tikzpicture}
        \global\coordinate (p1) at (1, 0);
        \global\coordinate (p2) at (2, 2);
        \global\coordinate (p2+) at (2.25, 1.5);
        \global\coordinate (p2m) at (2.5, 1);
        \global\coordinate (p3) at (3, 0);
        \global\coordinate (p3+) at (3.25, 0.5);
        \global\coordinate (p4-) at (3.75, 1.5);
        \global\coordinate (p4) at (4, 2);
        \global\coordinate (p4m) at (4.5, 1);
        \global\coordinate (p5-) at (4.75, 0.5);
        \global\coordinate (p5) at (5, 0);
        \global\coordinate (p6) at (6, 2);
        \global\coordinate (v_T) at (3, 3);
        \global\coordinate (v_T') at (3.5, 3.5);

            \draw[fill] (p1) circle (2pt);
            \draw[fill] (p2) circle (2pt);
            \draw[fill] (p2+) circle (2pt);
            \draw[fill] (p3) circle (2pt);
            \draw[fill] (p4) circle (2pt);
            \draw[fill] (p5-) circle (2pt);
            \draw[fill] (p5) circle (2pt);
            \draw[fill] (p6) circle (2pt);

            \draw (p1) -- (p2);
            \draw (p1) -- (p2+);
            \draw (p2) -- (p3);
            \draw (p3) -- (p4);
            \draw (p4) -- (p5);
            \draw (p5) -- (p6);
            \draw (p5-) -- (p6);
        \end{tikzpicture}    
    \caption{An example of $\alpha$ when $\alpha_1\neq 0$ and $s$ is odd}
    \label{fig:alphaexdoubly}
\end{figure}

Define \( P_T^{(1)(n+1)}(\alpha) \) to be the poset obtained by adding a new vertex \( v_T \) that lies above both \( v_{\alpha_1 + 1} \) and \( v_{n+1} \) in the fence poset associated with \( \alpha \). Let \( R_T^{(1)(n+1)}(\alpha; q) \) denote its rank polynomial. The ideals of \( P_T^{(1)(n+1)}(\alpha) \) can be divided into two types: those that contain \( v_T \) and those that do not.

Let \( \beta \) be the composition obtained by removing \( v_T \) and all vertices below it, and let \( d_1 \) be the number of vertices below \( v_T \) in the poset \( P_T^{(1)(n+1)}(\alpha) \). Then:
\[
R_T^{(1)(n+1)}(\alpha; q) = R^{(1)(n+1)}(\alpha; q) + q^{d_1 + 1} R(\beta; q).
\]

Next, define \( \overline{P_T^{(n+1)}}(\alpha) \) as the poset obtained by removing the relation \( v_{\alpha_1+2} \prec v_{\alpha_1 + 1} \) from \( P_T^{(1)(n+1)}(\alpha) \), and let \( \overline{R_T^{(n+1)}}(\alpha; q) \) be its rank polynomial.

Let \( \gamma^{(n+1)} \) be the poset formed by removing \( v_{\alpha_1 + 2} \) and all vertices above it, as well as removing \( v_{\alpha_1 + 1} \) and all vertices below it, from \( \overline{P_T^{(n+1)}}(\alpha) \). Then, the number of vertices below \( v_{\alpha_1 + 1} \) in \( \overline{P_T^{(n+1)}}(\alpha) \) is $\alpha_1$. Hence,
\[
\overline{R_T^{(n+1)}}(\alpha; q) = R_T^{(1)(n+1)}(\alpha; q) + q^{\alpha_1 + 1} R(\gamma^{(n+1)}; q),
\]
which implies:
\[
R^{(1)(n+1)}(\alpha; q) = \overline{R_T^{(n+1)}}(\alpha; q) - q^{d_1 + 1} R(\beta; q) - q^{\alpha_1 + 1} R(\gamma^{(n+1)}; q).
\]

Now we decompose \( \overline{P_T^{(n+1)}}(\alpha) \). Let \( \overline{P_T}(\alpha) \) be the circular fence poset obtained by removing the relation \( v_{n - \alpha_s} \prec v_{n+1} \) from \( \overline{P_T^{(n+1)}}(\alpha) \), and denote its rank polynomial by \( \overline{R_T}(\alpha; q) \). Let \( \gamma_1 \) be the poset formed by removing \( v_{n - \alpha_s} \) and all vertices above it, and \( v_{n+1} \) and all vertices below it, from \( \overline{P_T}(\alpha) \). Then, the number of vertices below \( v_{n+1} \) in \( \overline{P_T}(\alpha) \) equals to $\alpha_s$. Therefore,
\[
\overline{R_T}(\alpha; q) = \overline{R_T^{(n+1)}}(\alpha; q) + q^{\alpha_s + 1} R(\gamma_1; q).
\]

\begin{table}[H]
    \centering
    \begin{tabular}{ccc}
        \begin{tikzpicture}[scale=0.8, transform shape]
        \global\coordinate (p1) at (1, 0);
        \global\coordinate (p2) at (2, 2);
        \global\coordinate (p2+) at (2.25, 1.5);
        \global\coordinate (p2m) at (2.5, 1);
        \global\coordinate (p3) at (3, 0);
        \global\coordinate (p3+) at (3.25, 0.5);
        \global\coordinate (p4-) at (3.75, 1.5);
        \global\coordinate (p4) at (4, 2);
        \global\coordinate (p4m) at (4.5, 1);
        \global\coordinate (p5-) at (4.75, 0.5);
        \global\coordinate (p5) at (5, 0);
        \global\coordinate (p6) at (6, 2);
        \global\coordinate (v_T) at (3, 3);
        \global\coordinate (v_T') at (3.5, 3.5);

            \draw[fill] (p1) circle (2pt);
            \draw[fill] (p2) circle (2pt);
            \draw[fill] (p2+) circle (2pt);
            \draw[fill] (p3) circle (2pt);
            \draw[fill] (p4) circle (2pt);
            \draw[fill] (p5-) circle (2pt);
            \draw[fill] (p5) circle (2pt);
            \draw[fill] (p6) circle (2pt);
            \draw[fill] (v_T) node[above] {$v_T$} circle (2pt);

            \draw (p1) -- (p2);
            \draw (p1) -- (p2+);
            \draw (p2) -- (p3);
            \draw (p3) -- (p4);
            \draw (p4) -- (p5);
            \draw (p5) -- (p6);
            \draw (p5-) -- (p6);
            \draw (p2) -- (v_T);
            \draw (p6) -- (v_T);
        \end{tikzpicture} & 
                \begin{tikzpicture}[scale=0.8, transform shape]
            \draw[fill] (p1) circle (2pt);
            \draw[fill] (p2) circle (2pt);
            \draw[fill] (p2+) circle (2pt);
            \draw[fill] (p3) circle (2pt);
            \draw[fill] (p4) circle (2pt);
            \draw[fill] (p5-) circle (2pt);
            \draw[fill] (p5) circle (2pt);
            \draw[fill] (p6) circle (2pt);
            \draw[fill] (v_T) node[above] {$v_T$} circle (2pt);
            
            \draw (p1) -- (p2);
            \draw (p1) -- (p2+);
            \draw (p2+) -- (p3);
            \draw (p3) -- (p4);
            \draw (p4) -- (p5);
            \draw (p5) -- (p6);
            \draw (p5-) -- (p6);
            \draw (p2) -- (v_T);
            \draw (p6) -- (v_T);
        \end{tikzpicture} &
        \begin{tikzpicture}[scale=0.8, transform shape]
            \draw[fill] (p1) circle (2pt);
            \draw[fill] (p2) circle (2pt);
            \draw[fill] (p2+) circle (2pt);
            \draw[fill] (p3) circle (2pt);
            \draw[fill] (p4) circle (2pt);
            \draw[fill] (p5) circle (2pt);
            \draw[fill] (p6) circle (2pt);
            \draw[fill] (v_T) node[above] {$v_T$} circle (2pt);
            
            \draw (p1) -- (p2);
            \draw (p1) -- (p2+);
            \draw (p2+) -- (p3);
            \draw (p3) -- (p4);
            \draw (p4) -- (p5);
            \draw (p5) -- (p6);
            \draw (p2) -- (v_T);
            \draw (p6) -- (v_T);
        \end{tikzpicture}
        \\
        $P_T^{(1) (n+1)} (\alpha)$ & 
        $\overline{P^{(n+1)}_T}(\alpha)$ &
        $\overline{P_T}(\alpha)$
        \\
    \end{tabular}
    \caption{An example of $P_T^{(1)(n+1)}(\alpha)$, $\overline{P^{(n+1)}_T}(\alpha)$, and $\overline{P_T}(\alpha)$ for the instance $\alpha$ shown in Figure~\ref{fig:alphaexdoubly}}
    \label{tab:fence_posets}
\end{table}

We decompose \( \gamma^{(n+1)} \). Let \( \gamma_2 \) be the poset obtained by removing \( v_T \) and all vertices below it, and let \( d_2 \) be the number of such vertices. Define \( \gamma_3^{(n+1)} \) as the singly notched poset obtained from \( \gamma^{(n+1)} \) by removing only \( v_T \). Then:
\[
R(\gamma^{(n+1)}; q) = q^{d_2 + 1} R(\gamma_2; q) + R(\gamma_3^{(n+1)}; q).
\]

Hence, we get
\begin{eqnarray*}
    R^{(1)(n+1)}(\alpha; q) &= \overline{R_T}(\alpha; q) - q^{d_1 + 1} R(\beta; q) - q^{\alpha_s+1}R(\gamma_1;q)\\
    &- q^{\alpha_1+d_2 + 2}  R(\gamma_2; q) - q^{\alpha_1+1}R(\gamma_3^{(n+1)}; q).
\end{eqnarray*}

Then, we define the poset \( \delta \) as follows. Take the union of \( \beta \), \( \gamma_1 \), and \( \gamma_2 \), preserving all partial orders. One terminal vertex in this union is \( v_T \), and let \( v_P \), for some \( 1 \leq P \leq n - \alpha_s \), be another endpoint. Add a vertex \( v_{P+1} \), incorporating the relation between \( v_P \) and \( v_{P+1} \) as it appears in \( \alpha \). Add a relation \( v_{T} \prec v_{P+1} \) and relabel \(v_{P+1}\) as \( v_{\tilde{T}} \), resulting in the circular fence poset \( \delta \).

There are two types of ideals in \( \delta \):
\begin{itemize}
    \item Ideals that contain \( v_{\tilde{T}} \), which split into two subtypes:
        \begin{itemize}
            \item those that contain \( v_{\alpha_1 + 2} \), corresponding to \( R(\beta; q) \),
            \item and those that exclude \( v_{\alpha_1 + 2} \), corresponding to \( R(\gamma_2; q) \).
        \end{itemize}
    \item Ideals that do not contain \( v_{\tilde{T}} \), corresponding to \( R(\gamma_1; q) \).
\end{itemize}

\begin{figure}[H]
    \centering
    \begin{tikzpicture}
        \global\coordinate (p1) at (1, 0);
        \global\coordinate (p2) at (2, 2);
        \global\coordinate (p2+) at (2.25, 1.5);
        \global\coordinate (p3) at (3, 0);
        \global\coordinate (p4) at (4, 2);
        \global\coordinate (p4m) at (4.5, 1);
        \global\coordinate (p5-) at (4.75, 0.5);
        \global\coordinate (v_T) at (3, 3);

        \draw[fill] (p1) circle (2pt);
        \draw[fill] (p2) circle (2pt);
        \draw[fill] (p2+) circle (2pt);
        \draw[fill] (p3) circle (2pt);
        \draw[fill] (p4) circle (2pt);
        \draw[fill] (p2+) node[above right] {$v_{\alpha_1+2}$} circle (2pt);
        \draw[fill] (p4m) node[below left] {$v_{P}$} circle (2pt);
        \draw[fill] (p5-) node[below right] {$v_{P+1}$} circle (2pt);
        \draw[fill] (v_T) node[below right] {$v_T$} circle (2pt);
        
        \draw (p1) -- (p2);
        \draw (p1) -- (p2+);
        \draw (p2+) -- (p3);
        \draw (p3) -- (p4);
        \draw (p4) -- (p5-);
        \draw (p2) -- (v_T);

        \draw(0.875, -0.25) rectangle (4.625, 3.25);
    \end{tikzpicture}
    \caption{Constructing $\delta$ for $\alpha$ in Figure~\ref{fig:alphaexdoubly}}
    \label{fig:deltadoubly}
\end{figure}
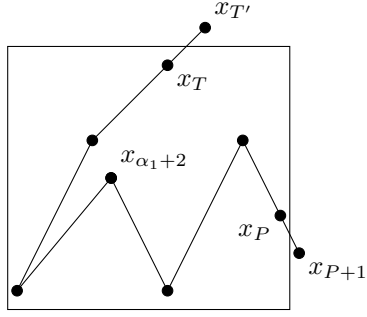

Thus, we obtain the relation:
\[
R^{(1)(n+1)}(\alpha; q) = \overline{R_T}(\alpha; q) - q^{\alpha_s + 1} \overline{R}(\delta; q) - q^{\alpha_1 + 1} R(\gamma_3^{(n+1)}; q).
\]

By Theorem~\ref{thm:circular}, \( \overline{R_T}(\alpha; q) \) is a symmetric polynomial, and by Theorem~\ref{thm:RPFP}, \( R(\gamma_3^{(n+1)}; q) \) is unimodal. Note that:
\[
\deg(\overline{R_T}(\alpha; q)) = n + 2, \quad \deg(\overline{R}(\delta; q)) = n - \alpha_s + 1, \quad \deg(R(\gamma_3^{(n+1)}; q)) = n - \alpha_1.
\]

Let their rank sequences be:
\[
(a_0, a_1, \ldots, a_{n+2}), \quad (b_0, b_1, \ldots, b_{n - \alpha_s + 1}), \quad (c_0, c_1, \ldots, c_{n - \alpha_1}).
\]

    Since \( \overline{R_T}(\alpha;q) \) is symmetric, we have:
    \[
    a_i = a_{n + 2 - i} \quad \text{for all } i.
    \]

    Moreover, since \( b_{n - \alpha_s + 1 - i} \geq 0 \) and \( c_{n - \alpha_1 + 1 - i} \geq 0 \) and from Lemma~\ref{lem:left} and Corollary~\ref{cor:eqauni}, it follows that:
    \[
    r_i = a_i - b_{i - \alpha_s - 1} - c_{i - \alpha_1 - 1} \geq a_{n + 2 - i} - b_{n - \alpha_s + 1 - i} - c_{n - \alpha_1 + 1 - i} = r_{n + 2 - i}
    \]
    for all \( i \leq \left\lfloor \frac{n+2}{2} \right\rfloor \), with the convention \( b_k = c_k = 0 \) for \( k < 0 \), and \( c_{n - \alpha_1 + 1} = 0 \).

    Hence, the rank sequence \( r_i \) satisfies:
    \[
    r_i \geq r_{n + 2 - i} \quad \text{for } i = 0,1, 2, \ldots, \left\lfloor \frac{n+2}{2} \right\rfloor.
    \]

    Similarly, let \( P_B^{(1)(n+1)}(\alpha) \) be the poset obtained by adding a new vertex \( v_B \) lying below both \( v_1 \) and \( v_{n-\alpha_s+1} \) in the poset \( P^{(1)(n+1)}(\alpha) \). Let \( R_B^{(1)(n+1)}(\alpha; q) \) denote its rank polynomial. 

    Let \( \beta' \) be the composition obtained from \( \alpha \) by removing \( v_B \) and all vertices above it.  Let \( d'_1 \) denote the number of vertices above \( v_B \). Then we have:
    \[
    R_B^{(1)(n+1)}(\alpha; q) = qR^{(1)(n+1)}(\alpha; q) +  R(\beta'; q).
    \]

    Next, let \( \overline{P_B^{(n+1)}}(\alpha) \) be the poset obtained by removing the relation \( v_{1} \prec v_{\alpha_1 + 2} \) from \( P_B^{(1)(n+1)}(\alpha) \), and let \( \overline{R_B^{(n+1)}}(\alpha; q) \) denote its rank polynomial.

    Let \( \gamma'^{(n+1)} \) be the poset obtained from \( \overline{P_B^{(n+1)}}(\alpha) \) by removing \( v_{\alpha_1+2} \) and all vertices below it, and removing \( v_{1} \) along with all vertices above it. Then, the number of vertices below \( v_{\alpha_1+2} \) in $\gamma'^{(n+1)}$ is $\alpha_2-1$. The ideals of \( \overline{P_B^{(n+1)}}(\alpha) \) then consist of:
    - those in \( P_B^{(1)(n+1)}(\alpha) \), and
    - additional ideals from \( \gamma'^{(n+1)} \), which correspond to including \( v_{\alpha_1+2}\) but not \( v_{1} \), weighted by \( q^{\alpha_2} \).

 By applying the same decomposition argument as before, we obtain:
\[\overline{R^{(n+1)}_B}(\alpha; q) = R_B^{(1)(n+1)}(\alpha; q) + q^{\alpha_2} R(\gamma'^{(n+1)}; q)\] 

Combining the above, we have:
\[
q R^{(n+1)}(\alpha; q) = \overline{R_B^{(n+1)}}(\alpha; q) - R(\beta'; q) - q^{\alpha_2} R(\gamma'^{(n+1)}; q).
\]

We decompose \( \overline{P_B^{(n+1)}}(\alpha) \). Let \( \overline{P_B}(\alpha) \) be the circular fence poset obtained by removing the relation \( v_{n-\alpha_s+1} \prec v_{n-\alpha_s} \) from \( \overline{P_B^{(n+1)}}(\alpha) \), and denote its rank polynomial by \( \overline{R_B}(\alpha; q) \). Let \( \gamma_1' \) be the poset formed by removing \( v_{n -\alpha_s+1} \) and all vertices above it, and \( v_{n-\alpha_s} \) and all vertices below it, from \( \overline{P_B}(\alpha) \). Let $d'_2$ be the number of vertices below $v_{n-\alpha_s}$. Then:
\[
\overline{R_B}(\alpha; q) = \overline{R_B^{(n+1)}}(\alpha; q) + q^{d'_2 + 1} R(\gamma'_1; q).
\]

We now decompose \( \gamma'^{(n+1)} \). Let \( \gamma'_2 \) be the poset obtained by removing \( v_B \) and all vertices above it. Then, the number of such vertices equals to $\alpha_{s-1}+\alpha_s+2$. Define \( \gamma_3'^{(n+1)} \) as the singly notched poset obtained from \( \gamma^{(n+1)} \) by removing only \( v_B \). Then:
\[
R(\gamma'^{(n+1)}; q) =  R(\gamma'_2; q) + qR(\gamma_3'^{(n+1)}; q).
\]

We define the poset \( \delta' \) as follows. Take the union of \( \beta' \), \( \gamma_1' \), and \( \gamma_2' \), preserving all partial orders. One terminal vertex in this union is \( v_B \), and let \( v_Q \), for some \( 1 \leq Q \leq n - \alpha_s \), be another endpoint. Add a vertex \( v_{Q+1} \), incorporating the relation between \( v_Q \) and \( v_{Q+1} \) as it appears in \( \alpha \). Add a relation \( v_{B} \succ v_{Q+1} \) and relabel $v_{Q+1}$ as \( v_{\tilde{B}} \), resulting in the circular fence poset \( \delta' \).

There are two types of ideals in \( \delta' \):
\begin{itemize}
    \item Ideals that do not contain \( v_{\tilde{B}} \), which split into two subtypes:
        \begin{itemize}
            \item those that do not contain \( v_{n-\alpha_s} \), corresponding to \( R(\beta'; q) \),
            \item and those that contain \( v_{\alpha_1 + 2} \), corresponding to \( R(\gamma'_2; q) \).
        \end{itemize}
    \item Ideals that contain \( v_{\tilde{B}} \), corresponding to \( R(\gamma'_1; q) \).
\end{itemize}

Thus, we obtain the relation:
\[
R^{(1)(n+1)}(\alpha; q) = \overline{R_B}(\alpha; q) - \overline{R}(\delta'; q) - q R(\gamma_3'^{(n+1)}; q).
\]

Now, observe that \( \deg(\overline{R_B}(\alpha; q)) = n + 2 \), \( \deg(\overline{R}(\delta'; q)) = n - \alpha_s + 1\), and \(\deg(qR(\gamma_3'^{(n+1)};q))=n-\alpha_1-\alpha_2-1\). 

Let their rank sequences be:
\[
(t_0, t_1, \ldots, t_{n+2}), \quad (u_0, u_1, \ldots, u_{n - \alpha_s + 1}), \quad (w_0, w_1, \ldots, w_{n - \alpha_1-\alpha_2-1}).
\]

    Since \( \overline{R_B}(\alpha;q) \) is symmetric, we have:
    \[
    t_i = t_{n + 2 - i} \quad \text{for all } i.
    \]

Using the previous argument with Lemma~\ref{lem:left} and Corollary~\ref{cor:eqauni}, the rank sequence \( r_i \) satisfies:
    \[
    r_i \leq r_{n  - i} \quad \text{for } i = 0,1, 2, \ldots, \left\lfloor \frac{n-1}{2} \right\rfloor.
    \]

\end{proof}

\begin{lemma}\label{lem:doublyeqb}
    The rank sequence \( r(\alpha) = (r_{0}, r_{1}, \ldots, r_{n+1}) \) of the doubly notched poset $P^{(1)(n+1)}(\alpha)$ satisfies \eqref{eq:ineqB}.
\end{lemma}
\begin{proof}
We first treat the case when \( \alpha_1\neq 0 \). The ideals of the doubly notched poset \( P^{(1)(n+1)}(\alpha) \) can be partitioned into those that contain \( v_{\alpha_1+2} \) and those that do not. Let \( \beta_1 \) be the singly notched poset obtained by removing \( v_{\alpha_1+2} \) along with all elements below it, and let \( \beta_2 \) be the singly notched poset obtained by removing only \( v_{\alpha_1+2} \). Then we have the decomposition
\[
R^{(1)(n+1)}(\alpha; q) = q^{d_1 + 1} R^{(n+1)}(\beta_1; q) + R^{(n+1)}(\beta_2; q),
\]
where \( d_1 \) is the number of elements lying below \( v_{\alpha_1+2} \).

Write
\[
R^{(n+1)}(\beta_1; q) = \sum_{i=0}^{n - d_1} b_i q^i, \quad R^{(n+1)}(\beta_2; q) = \sum_{i=0}^{n} b_i' q^i.
\]
By Lemma~\ref{lem:sineqB}, we have the following inequalities:
\begin{align*}
    &b_0 \le b_1 \le \dots \le b_{\lfloor (n - d_1)/2 \rfloor}, \\
    &b_0' \le b_1' \le \dots \le b_{\lfloor n/2 \rfloor}'.
\end{align*}

Since \( q^{d_1 + 1} R^{(n+1)}(\beta_1; q) \) begins at degree \( d_1 + 1 \), and \( R(\beta_2; q) \) starts at degree zero, the coefficients of their sum increase up to \( \lfloor n/2 \rfloor \), so
\[
r_0 \le r_1 \le \dots \le r_{\lfloor n/2 \rfloor}.
\]

Now consider the other inequality. Again, partition the ideals of \( P^{(1)(n+1)}(\alpha) \) into those that contain the element \( v_{1} \) and those that do not. Let \( \gamma_1 \) be the singly notched poset obtained by removing \( v_{1} \) along with all elements above it, and let \( \gamma_2 \) be obtained by removing only \( v_{1} \). Then we have
\[
R^{(1)(n+1)}(\alpha; q) = R^{(n+1)}(\gamma_1; q) + q R^{(n+1)}(\gamma_2; q).
\]

Write
\[
R^{(n+1)}(\gamma_1; q) = \sum_{i=0}^{n - d_2} c_i q^i, \quad R^{(n+1)}(\gamma_2; q) = \sum_{i=0}^{n} c_i' q^i,
\]
where \( d_2 \) is the number of elements above \( v_{n - \alpha_s + 1} \). Then, by Lemma~\ref{lem:sineqB},
\begin{align*}
    &c_{\lceil (n - d_2)/2 \rceil} \ge c_{\lceil (n - d_2)/2 \rceil + 1} \ge \dots \ge c_{n - d_2}, \\
    &c'_{\lceil n/2 \rceil} \ge c'_{\lceil n/2 \rceil + 1} \ge \dots \ge c'_n.
\end{align*}

Since \( q R(\gamma_2; q) \) starts at degree one, and the coefficients from both terms overlap starting around \( \lceil n/2 \rceil + 1 \), we obtain
\[
r_{\lceil n/2 \rceil + 1} \ge r_{\lceil n/2 \rceil + 2} \ge \dots \ge r_{n+1},
\]
which establishes the decreasing part of the sequence.

For \( \alpha_1=0 \) even, the argument is entirely analogous by changing $v_{\alpha_1+2}$ to $v_1$ for the first inequality and changing \( v_{1} \) to \( v_{\alpha_2+2} \) for the second inequality; In both cases, the rank sequence satisfies inequality~\textnormal{\eqref{eq:ineqB}}.
\end{proof}
\begin{theorem}
    The rank polynomials of doubly notched posets are almost interlacing
\end{theorem}

\begin{proof}
Lemma~\ref{lem:doublyeqa}, Lemma~\ref{lem:doublyeqb}, and Lemma~\ref{lem:ABau} implies the statement.
\end{proof}

We rephrase above theorems in cluster algebra language as follows.

\begin{theorem}
    The rank polynomial \( R(P_\gamma; q) \) is almost interlacing when \( \gamma \) is an arc in (un)punctured surface.
\end{theorem}

\section{Single Lamination}\label{sec:singlelam}

Recall that when the lamination consists of a single curve, we refer to it as a single lamination. In this case, the corresponding coefficient is not a principal coefficient, and there is no natural poset interpretation. Therefore, we introduce a new terminology for the associated polynomial.  

\subsection{Unimodality and Log-concavity}

In this section, we do not apply elementary laminations. Hence, the $y_i$ variables do not represent principal coefficients. Nonetheless, we maintain the notation $y_i$ to represent the shear coordinates associated with a single lamination.

\begin{definition}
    For the cluster expansion $x_\gamma$, we define the \emph{$c$-polynomial} $c_{x_\gamma}(q)$ by setting $x_i = 1$ for all $i$ and identifying all $y_i$'s with $q$, i.e.,
    \[
    c_{x_\gamma}(q) = x_\gamma(1,\ldots,1,q,\ldots,q).
    \]
\end{definition}

Note that the rank polynomial introduced in the previous sections is also a $c$-polynomial. i.e. rank polynomials are special cases of $c$-polynomials.

\begin{example}
Let's revisit Example~\ref{ex:single}. Note that the exchange matrix is given in Example~\ref{ex:single}. 
\begin{figure}[H]
    \centering
    \begin{tikzpicture}
    \draw[line width=0.5mm] (0,0) to (1.4,0);
    \draw[line width=0.5mm] (1.4,0) to (2.4,1);
    \draw[line width=0.5mm] (2.4,2.4) to (2.4,1);
    \draw[line width=0.5mm] (2.4,2.4) to (1.4,3.4);
    \draw[line width=0.5mm] (0,3.4) to (1.4,3.4);
    \draw[line width=0.5mm] (0,3.4) to (-1,2.4);
    \draw[line width=0.5mm] (-1,1) to (0,0);
    \draw[line width=0.5mm] (-1,1) to (-1,2.4);

    \draw[line width=0.5mm] (1.4,0) to (2.4,2.4);
    \draw[line width=0.5mm] (1.4,3.4) to (1.4,0);
    \draw[line width=0.5mm] (1.4,3.4) to (0,0);
    \draw[line width=0.5mm] (0,3.4) to (0,0);
    \draw[line width=0.5mm] (-1,2.4) to (0,0);

    \node[scale=0.8] at (-0.7,1.3) {$1$};
    \node[scale=0.8] at (-0.2,1.5) {$2$};
    \node[scale=0.8] at (0.7,1.3) {$3$};
    \node[scale=0.8] at (1.6,2.3) {$4$};
    \node[scale=0.8] at (2.1,1.3) {$5$};

    \draw[line width=0.2mm, out=30,in=200,looseness=1, red] (-1,1.7) to (2.4,1.7);

    \node[scale=0.7, red] at (0.4,2) {$L$};
    \draw[line width=0.5mm, blue] (-1,1) to (2.4,1);
    \node[blue] at (1,0.7) {$\gamma$};

    \filldraw[black] (0,0) circle (2pt);
    \filldraw[black] (1.4,0) circle (2pt);
    \filldraw[black] (2.4,1) circle (2pt);
    \filldraw[black] (2.4,2.4) circle (2pt);
    \filldraw[black] (1.4,3.4) circle (2pt);
    \filldraw[black] (0,3.4) circle (2pt);
    \filldraw[black] (-1,2.4) circle (2pt);
    \filldraw[black] (-1,1) circle (2pt);
    \end{tikzpicture}
    \caption{Triangulation on Example~\ref{ex:single} with an arc $\gamma$ added}
\end{figure}

To compute the $c$-polynomial, we apply a sequence of flips (mutations) to the initial triangulation. For the arc $\gamma$, the resulting $c$-polynomial is given by:
\[
c_{x_{\gamma}}(q) = 2q^2 + 2q + 2.
\]
\end{example}

\begin{example}

Figure~\ref{fig:excpoly} shows a triangulation $T$ of an octagon with a single lamination $L$, where both triangulation and lamination being different from Example~\ref{ex:single}.
    \begin{figure}[H]
    \centering
\begin{tikzpicture}[scale=0.6, transform shape]
    \def\n{8}
    \def\radius{3}

    \foreach \i in {1,...,\n} {
        \pgfmathsetmacro{\angle}{360/\n * \i + 50}
        \coordinate (V\i) at (\angle:\radius);
    }

    \foreach \i in {1,...,\n} {
        \pgfmathsetmacro{\nexti}{mod(\i, \n) + 1}
        \draw[ultra thick] (V\i) -- (V\nexti);
    }

    \draw[thick] (V3) -- (V1) node[pos=0.4, yshift=-1pt, right] {$1$};
    \draw[thick] (V3) -- (V8) node[pos=0.35, yshift=-4pt, right] {$2$};
    \draw[thick] (V4) -- (V8) node[pos=0.4, yshift=-2pt, right] {$3$};
    \draw[thick] (V5) -- (V8) node[pos=0.4, right] {$4$};
    \draw[thick] (V6) -- (V8) node[pos=0.4, right] {$5$};
    
    \draw[line width=0.5mm, blue] (V2) -- (V7) node[pos=0.4, above right] {$\gamma$};

    \coordinate (M1) at ($0.5*(V1) + 0.5*(V2)$);
    \coordinate (M2) at ($0.5*(V5) + 0.5*(V6)$);

    \draw[red, looseness=1] (M1) .. controls +(up:-1.5) and +(down:-1.5) .. (M2) node[pos=0.65, xshift=-2pt, below] {$L$};
\end{tikzpicture}

    \caption{Triangulation of a polygon with a single lamination}
    \label{fig:excpoly}
\end{figure}

Here, we denote the cluster variable associated with each arc $i$ as $x_i$. The extended exchange matrix is given by:
\[
\begin{bmatrix}
    0 & -1 & 0 & 0 & 0 \\
    1 & 0 & 1 & 0 & 0 \\
    0 & -1 & 0 & 1 & 0 \\
    0 & 0 & -1 & 0 & 1 \\
    0 & 0 & 0 & -1 & 0 \\ 
    -1 & 1 & 0 & -1 & 0
\end{bmatrix},
\]
where the final row corresponds to the shear coordinates $(b_\tau(T,L))_{\tau \in T}$. For the arc $\gamma$, the corresponding $c$-polynomial is calculated as:
\[
c_{x_{\gamma}}(q) = 2q^3 + 4q^2 + 2q + 1.
\]
\end{example}

\begin{theorem}\label{thm:uni_single}
If the lamination $L$ is a single lamination and $(S,M)$ is an unpunctured disk, then $c_{x_\gamma}(q)$ is unimodal.
\end{theorem}

To prove Theorem~\ref{thm:uni_single}, we first establish the combinatorial framework for the arc $\gamma$. We denote the starting endpoint of $\gamma$ as $s(\gamma)$. Let $\{\tau_1, \tau_2, \dots, \tau_j\}$ be the sequence of arcs in the initial triangulation $T$ intersected by $\gamma$, ordered by intersection starting from $s(\gamma)$. 

There are two primary geometric configurations based on the local behavior of these intersections: The first case occurs when a subsequence of arcs $(\tau_{i-1}, \tau_i, \tau_{i+1})$ forms a zigzag pattern (either an $S$-shape or a $Z$-shape). The second case encompasses all remaining configurations where this zigzag structure is absent. These scenarios are illustrated in the figure below, with the $Z$ configurations on the left and the non-zigzag cases on the right. For any $k$, let $\gamma_k$ denote the arc originating at $s(\gamma)$ that consecutively intersects $\tau_1$ through $\tau_k$.

\begin{center}
\begin{tabular}{ccc}\label{table:two}
\begin{tikzpicture}[scale=0.45, transform shape]
    \draw (2,-5) to (3,0);
    \draw (2,-5) to (1,3);
    \draw (3,0) to (1,3);
    \draw (2,-5) to (-1,3);
    \draw (1,3) to (-1,3);
    \draw (-1,-6) to (-1,3);
    \draw (-1,-6) to (2,-5);

    \draw[blue, thick] (-8,0) to (2,-5);
    \draw[blue, thick] (-8,0) to (1,3);
    \draw[blue, thick] (-8,0) to (-1,3);    
    
    \node[scale=2, fill=white] at (0.5,-1) {$\tau_k$};
    \node[scale=2, fill=white] at (-1.5,0) {$\tau_{k-1}$};
    \node[scale=2, fill=white] at (1.6,0) {$\tau_{k+1}$};

    \node[scale=2, blue, fill=white] at (-4,2) {$\gamma_h$};
    \node[scale=2, blue, fill=white] at (-3,1.5) {$\gamma_k$};
    \node[scale=2, blue, fill=white] at (-3,-2.5) {$\gamma_{k-1}$};

    \node[scale=2] at (-8,-1) {$s(\gamma)$};

    \draw[fill=black] (2,-5) circle [radius=2pt];
    \draw[fill=black] (3,0) circle [radius=2pt];
    \draw[fill=black] (1,3) circle [radius=2pt];
    \draw[fill=black] (-1,3) circle [radius=2pt];
    \draw[fill=black] (-1,-6) circle [radius=2pt];
    \draw[fill=black] (-8,0) circle [radius=2pt];

\end{tikzpicture}
&
\qquad \qquad
&
\begin{tikzpicture}[scale=0.4, transform shape]
    \draw (2,-5) to (3,0);
    \draw (2,-5) to (1,3);
    \draw (3,0) to (1,3);
    \draw (2,-5) to (-1,4);
    \draw (1,3) to (-1,4);
    \draw (-4,5) to (2,-5);
    \draw (-4,5) to (-1,4);
    \draw (-1,-6) to (2,-5);
    \draw (-8,4) to (-4,5);
    \draw (-8,4) to (2,-5);
    \draw (-8,4) to (-1,-6);

    \draw[blue, thick] (-10,0) to (2,-5);
    \draw[blue, thick] (-10,0) to (1,3);
    \draw[blue, thick] (-10,0) to (-1,4);    

    \node[scale=1.5, fill=white] at (-2.7,0.5) {$\iddots$};

    \node[scale=2, fill=white] at (0.5,-1) {$\tau_k$};
    \node[scale=2, fill=white] at (-1.2,0) {$\tau_{k-1}$};
    \node[scale=2, fill=white] at (1.6,0) {$\tau_{k+1}$};
    \node[scale=2, fill=white] at (-2.3,-1) {$\tau_{h+1}$};
    \node[scale=2, fill=white] at (-4.5,-1) {$\tau_h$};

    \node[scale=2] at (-10,-1) {$s(\gamma)$};

    \node[scale=2, blue, fill=white] at (-4,2.7) {$\gamma_{k-1}$};
    \node[scale=2, blue, fill=white] at (-3,2) {$\gamma_k$};

    \node[scale=2, blue, fill=white] at (-3,-3) {$\gamma_{h}$};

    \draw[fill=black] (2,-5) circle [radius=2pt];
    \draw[fill=black] (3,0) circle [radius=2pt];
    \draw[fill=black] (1,3) circle [radius=2pt];
    \draw[fill=black] (-1,4) circle [radius=2pt];
    \draw[fill=black] (-1,-6) circle [radius=2pt];
    \draw[fill=black] (-4,5) circle [radius=2pt];
    \draw[fill=black] (-8,4) circle [radius=2pt];
    \draw[fill=black] (-4,5) circle [radius=2pt];
    \draw[fill=black] (-10,0) circle [radius=2pt];

\end{tikzpicture}
\end{tabular}
\end{center}

Let $A_k$ be the exchange matrix of the flip of $\gamma_k$. Then, for $i<k$, $(A_k)_{ik}$ is nonzero for two $i$. One of the two $i$'s is $k-1$, and the other one is denoted by $h(k)$.

For the first case, depending on the lamination, we identify six cases for \( x_k \) and \( x_{k+1} \):
\begin{itemize}
    \item \( c_{x_{k+1}} = qc_{x_k} + c_{x_{k-1}} \) and \( c_{x_k} = c_{x_{k-1}} + c_{x_h} \)
    \item \( c_{x_{k+1}} = c_{x_k} + qc_{x_{k-1}} \) and \( c_{x_k} = qc_{x_{k-1}} + c_{x_h} \)
    \item \( c_{x_{k+1}} = c_{x_k} + qc_{x_{k-1}} \) and \( c_{x_k} = c_{x_{k-1}} + c_{x_h} \)
    \item \( c_{x_{k+1}} = c_{x_k} + c_{x_{k-1}} \) and \( c_{x_k} = qc_{x_{k-1}} + c_{x_h} \)
    \item \( c_{x_{k+1}} = c_{x_k} + c_{x_{k-1}} \) and \( c_{x_k} = c_{x_{k-1}} + qc_{x_h} \)
    \item \( c_{x_{k+1}} = c_{x_k} + c_{x_{k-1}} \) and \( c_{x_k} = c_{x_{k-1}} + c_{x_h} \)
\end{itemize}

\begin{table}[H]
\captionsetup{width=0.9\textwidth}
    \centering
    \begin{tabular}{ccc}\label{table:firstcase}
    \begin{tikzpicture}[scale=0.3, transform shape]
        \draw (2,-5) to (3,0);
        \draw (2,-5) to (1,3);
        \draw (3,0) to (1,3);
        \draw (2,-5) to (-1,3);
        \draw (1,3) to (-1,3);
        \draw (-1,-6) to (-1,3);
        \draw (-1,-6) to (2,-5);

        \draw[blue, thick] (-8,0) to (2,-5);
        \draw[blue, thick] (-8,0) to (1,3);
        \draw[blue, thick] (-8,0) to (-1,3);    
        
        \node[scale=2, fill=white] at (0.5,-1) {$\tau_k$};
        \node[scale=2, fill=white] at (-1.5,0) {$\tau_{k-1}$};
        \node[scale=2, fill=white] at (1.6,0) {$\tau_{k+1}$};
    
        \node[scale=2, blue, fill=white] at (-4,2) {$\gamma_h$};
        \node[scale=2, blue, fill=white] at (-3,1.5) {$\gamma_k$};
        \node[scale=2, blue, fill=white] at (-3,-2.5) {$\gamma_{k-1}$};

        \draw[red] (2.7,0.4) to [out = -100, in = 90] (0,-5.7);
        \node[scale=2] at (-8,-1) {$s(\gamma)$};
    
        \draw[fill=black] (2,-5) circle [radius=2pt];
        \draw[fill=black] (3,0) circle [radius=2pt];
        \draw[fill=black] (1,3) circle [radius=2pt];
        \draw[fill=black] (-1,3) circle [radius=2pt];
        \draw[fill=black] (-1,-6) circle [radius=2pt];
        \draw[fill=black] (-8,0) circle [radius=2pt];
    
    \end{tikzpicture}
    &
    \begin{tikzpicture}[scale=0.3, transform shape]
        \draw (2,-5) to (3,0);
        \draw (2,-5) to (1,3);
        \draw (3,0) to (1,3);
        \draw (2,-5) to (-1,3);
        \draw (1,3) to (-1,3);
        \draw (-1,-6) to (-1,3);
        \draw (-1,-6) to (2,-5);

        \draw[blue, thick] (-8,0) to (2,-5);
        \draw[blue, thick] (-8,0) to (1,3);
        \draw[blue, thick] (-8,0) to (-1,3);    
        
        \node[scale=2, fill=white] at (0.5,-1) {$\tau_k$};
        \node[scale=2, fill=white] at (-1.5,0) {$\tau_{k-1}$};
        \node[scale=2, fill=white] at (1.6,0) {$\tau_{k+1}$};
    
        \node[scale=2, blue, fill=white] at (-4,2) {$\gamma_h$};
        \node[scale=2, blue, fill=white] at (-3,1.5) {$\gamma_k$};
        \node[scale=2, blue, fill=white] at (-3,-2.5) {$\gamma_{k-1}$};
    
        \node[scale=2] at (-8,-1) {$s(\gamma)$};

        \draw[red] (2.85,-0.6) to [out = 110, in = -60] (-2.2,3);

        \draw[fill=black] (2,-5) circle [radius=2pt];
        \draw[fill=black] (3,0) circle [radius=2pt];
        \draw[fill=black] (1,3) circle [radius=2pt];
        \draw[fill=black] (-1,3) circle [radius=2pt];
        \draw[fill=black] (-1,-6) circle [radius=2pt];
        \draw[fill=black] (-8,0) circle [radius=2pt];
    
    \end{tikzpicture}
    &
    \begin{tikzpicture}[scale=0.3, transform shape]
        \draw (2,-5) to (3,0);
        \draw (2,-5) to (1,3);
        \draw (3,0) to (1,3);
        \draw (2,-5) to (-1,3);
        \draw (1,3) to (-1,3);
        \draw (-1,-6) to (-1,3);
        \draw (-1,-6) to (2,-5);

        \draw[blue, thick] (-8,0) to (2,-5);
        \draw[blue, thick] (-8,0) to (1,3);
        \draw[blue, thick] (-8,0) to (-1,3);    
        
        \node[scale=2, fill=white] at (0.5,-1) {$\tau_k$};
        \node[scale=2, fill=white] at (-1.5,0) {$\tau_{k-1}$};
        \node[scale=2, fill=white] at (1.6,0) {$\tau_{k+1}$};
    
        \node[scale=2, blue, fill=white] at (-4,2) {$\gamma_h$};
        \node[scale=2, blue, fill=white] at (-3,1.5) {$\gamma_k$};
        \node[scale=2, blue, fill=white] at (-3,-2.5) {$\gamma_{k-1}$};
    
        \node[scale=2] at (-8,-1) {$s(\gamma)$};

        \draw[red] (2.85,-0.6) to [out = 120, in = -60] (-0.2,3);

        \draw[fill=black] (2,-5) circle [radius=2pt];
        \draw[fill=black] (3,0) circle [radius=2pt];
        \draw[fill=black] (1,3) circle [radius=2pt];
        \draw[fill=black] (-1,3) circle [radius=2pt];
        \draw[fill=black] (-1,-6) circle [radius=2pt];
        \draw[fill=black] (-8,0) circle [radius=2pt];
    
    \end{tikzpicture}
    \\
        \begin{tikzpicture}[scale=0.3, transform shape]
        \draw (2,-5) to (3,0);
        \draw (2,-5) to (1,3);
        \draw (3,0) to (1,3);
        \draw (2,-5) to (-1,3);
        \draw (1,3) to (-1,3);
        \draw (-1,-6) to (-1,3);
        \draw (-1,-6) to (2,-5);

        \draw[blue, thick] (-8,0) to (2,-5);
        \draw[blue, thick] (-8,0) to (1,3);
        \draw[blue, thick] (-8,0) to (-1,3);    
        
        \node[scale=2, fill=white] at (0.5,-1) {$\tau_k$};
        \node[scale=2, fill=white] at (-1.5,0) {$\tau_{k-1}$};
        \node[scale=2, fill=white] at (1.6,0) {$\tau_{k+1}$};
    
        \node[scale=2, blue, fill=white] at (-4,2) {$\gamma_h$};
        \node[scale=2, blue, fill=white] at (-3,1.5) {$\gamma_k$};
        \node[scale=2, blue, fill=white] at (-3,-2.5) {$\gamma_{k-1}$};
    
        \node[scale=2] at (-8,-1) {$s(\gamma)$};

        \draw[red] (2.7,0.4) to [out = 150, in = -60] (-2.2,3);

        \draw[fill=black] (2,-5) circle [radius=2pt];
        \draw[fill=black] (3,0) circle [radius=2pt];
        \draw[fill=black] (1,3) circle [radius=2pt];
        \draw[fill=black] (-1,3) circle [radius=2pt];
        \draw[fill=black] (-1,-6) circle [radius=2pt];
        \draw[fill=black] (-8,0) circle [radius=2pt];
    
    \end{tikzpicture}
    &
    \begin{tikzpicture}[scale=0.3, transform shape]
        \draw (2,-5) to (3,0);
        \draw (2,-5) to (1,3);
        \draw (3,0) to (1,3);
        \draw (2,-5) to (-1,3);
        \draw (1,3) to (-1,3);
        \draw (-1,-6) to (-1,3);
        \draw (-1,-6) to (2,-5);

        \draw[blue, thick] (-8,0) to (2,-5);
        \draw[blue, thick] (-8,0) to (1,3);
        \draw[blue, thick] (-8,0) to (-1,3);    
        
        \node[scale=2, fill=white] at (0.5,-1) {$\tau_k$};
        \node[scale=2, fill=white] at (-1.5,0) {$\tau_{k-1}$};
        \node[scale=2, fill=white] at (1.6,0) {$\tau_{k+1}$};
    
        \node[scale=2, blue, fill=white] at (-4,2) {$\gamma_h$};
        \node[scale=2, blue, fill=white] at (-3,1.5) {$\gamma_k$};
        \node[scale=2, blue, fill=white] at (-3,-2.5) {$\gamma_{k-1}$};
    
        \node[scale=2] at (-8,-1) {$s(\gamma)$};

        \draw[red] (0.5,3) to [out = -100, in = 90] (-0.5,-5.8);

        \draw[fill=black] (2,-5) circle [radius=2pt];
        \draw[fill=black] (3,0) circle [radius=2pt];
        \draw[fill=black] (1,3) circle [radius=2pt];
        \draw[fill=black] (-1,3) circle [radius=2pt];
        \draw[fill=black] (-1,-6) circle [radius=2pt];
        \draw[fill=black] (-8,0) circle [radius=2pt];
    \end{tikzpicture}
    &
    \begin{tikzpicture}[scale=0.3, transform shape]
        \draw (2,-5) to (3,0);
        \draw (2,-5) to (1,3);
        \draw (3,0) to (1,3);
        \draw (2,-5) to (-1,3);
        \draw (1,3) to (-1,3);
        \draw (-1,-6) to (-1,3);
        \draw (-1,-6) to (2,-5);

        \draw[blue, thick] (-8,0) to (2,-5);
        \draw[blue, thick] (-8,0) to (1,3);
        \draw[blue, thick] (-8,0) to (-1,3);    
        
        \node[scale=2, fill=white] at (0.5,-1) {$\tau_k$};
        \node[scale=2, fill=white] at (-1.5,0) {$\tau_{k-1}$};
        \node[scale=2, fill=white] at (1.6,0) {$\tau_{k+1}$};
    
        \node[scale=2, blue, fill=white] at (-4,2) {$\gamma_h$};
        \node[scale=2, blue, fill=white] at (-3,1.5) {$\gamma_k$};
        \node[scale=2, blue, fill=white] at (-3,-2.5) {$\gamma_{k-1}$};
    
        \node[scale=2] at (-8,-1) {$s(\gamma)$};

        \draw[red] (2.85,-0.6) to [out = -120, in = 60] (-1.5,-5);

        \draw[fill=black] (2,-5) circle [radius=2pt];
        \draw[fill=black] (3,0) circle [radius=2pt];
        \draw[fill=black] (1,3) circle [radius=2pt];
        \draw[fill=black] (-1,3) circle [radius=2pt];
        \draw[fill=black] (-1,-6) circle [radius=2pt];
        \draw[fill=black] (-8,0) circle [radius=2pt];
    \end{tikzpicture}
    \end{tabular}
    \caption{The first six cases of $x_k$ and $x_{k+1}$. Note that these diagrams are representative configurations and do not illustrate every possible scenario for each case; distinct laminations may yield the same case classification.}
    \label{tab:placeholder}
\end{table}

For the second geometric configuration, there are also six cases for \( x_k \) and \( x_{k+1} \).


\begin{proof}[Proof of Theorem~\ref{thm:uni_single}]
For the first case we prove $c_{x_{k+1}}=qc_{x_k}+c_{x_{k-1}}$ and $c_{x_k}=c_{x_{k-1}}+c_{x_h}$.

Obviously, all of the coefficients of $c_{x_{k}}$ are nonnegative. For $f(q)=a_{n} q^n+a_{n-1}q^{n-1}+\dots+a_0$, let $(i,j)$ be called $2$-peak of $f(q)$ when $\min{a_{i},a_{j}}\ge a_{k}$ for $k\neq i,j$. In fact, if $f(q)$ is unimodal, then the components of $2$-peak of $f(q)$ always has difference $1$. We prove the theorem by induction on $i$. We use the following induction hypothesis :
\begin{enumerate}
\item $c_{x_{k-1}}$, $c_{x_h}$, and $c_{x_k}$ are unimodal.
\item $c_{x_{k-1}}$ and $c_{x_{h}}$ have the same $2$-peak or $c_{x_{k-1}}$ has $(i+1,i+2)$ $2$-peak and $c_{x_{h}}$ has $(i,i+1)$ $2$-peak.
\end{enumerate}
The base cases $c_{x_0}$ and $c_{x_1}$ satisfy the inductive hypothesis, as their degrees are strictly less than $2$. Specifically, $c_{x_1}$ is restricted to the values $2$ or $q+1$. It follows that $c_{x_2}$ must take the form $q(q+1)+1$ or $(q+1)+1$ or $q+1$. Since both resulting polynomials are unimodal, the base case is well-founded.

Now, suppose that $c_{x_k}(q)$ satisfy induction hypothesis. Let
\begin{align*}
c_{x_h}(q)&=a_{m}q^m+\dots+a_0\\
c_{x_{k-1}}(q)&=b_{m'}q^{m'}+\dots+b_0.
\end{align*}
First, assume that $c_{x_h}$ and $c_{x_{k-1}}$ have $(i,i+1)$ as a $2$-peak. Then, we get $a_{i+1}\ge a_{i+2}\ge \dots$, $b_{i+1}\ge b_{i+2}\ge \dots$, $a_{i}\ge a_{i-1}\ge \dots$, $b_{i}\ge b_{i-1}\ge \dots$, and $c_{x_{k-1}}+c_{x_h}$ has $(i,i+1)$ $2$-peak. By induction hypothesis, $c_{x_k}(q)=c_{x_{k-1}}+c_{x_h}$ is unimodal which has $(i,i+1)$ $2$-peak. Now we consider $c_{x_{k+1}}(q)=c_{m''}q^{m''}+\dots+c_0$. Since
\begin{align*}
c_{x_{k+1}}(q)&=(q+1)c_{x_{k-1}}+qc_{x_h}\\
&=\dots+(a_{i+1}+b_{i+1}+b_{i+2})q^{i+2}+(a_i+b_i+b_{i+1})q^{i+1}\\
&+(a_{i-1}+b_{i-1}+b_{i})q^{i}+\dots,
\end{align*}
$c_{i}\ge c_{i-1}\ge \dots$ and $c_{i+2}\ge c_{i+3}\ge \dots$. Moreover, \[
c_{i+1}-c_{i}=(a_{i}-a_{i-1})+(b_{i+1}-b_{i-1})\ge 0
\]
and \[
c_{i+1}-c_{i+3}=(a_{i}-a_{i+2})+(b_{i}+b{i+1}-b_{i+2}-b_{i+3})\ge 0
\]
since $c_{x_{k-1}}$, $c_{x_h}$, and $c_{x_{k}}$ have $(i,i+1)$ $2$-peak. Hence, $c_{x_{k+1}}$ is unimodal and has $(i,i+1)$ or $(i+1,i+2)$ $2$-peak which implies that $c_{x_k}$ and $c_{x_{k+1}}$ satisfy induction hypothesis 1 and 2.

Secondly, assume that $c_{x_{k-1}}$ has $(i+1,i+2)$ $2$-peak and $c_{x_h}$ has $(i,i+1)$ $2$-peak. Then, since $c_{x_k}=c_{x_h}+c_{x_{k-1}}$ is unimodal, $c_{x_k}$ has $(i,i+1)$ or $(i+1,i+2)$ $2$-peak.
\item[(i)] The case when $c_{x_k}$ has $(i,i+1)$ $2$-peak\\
We have $a_{i}+b_i\ge a_{i+2}+b_{i+2}$. Let $c_{x_{k+1}}=c_{m''}q^{m''}+\dots+c_0$. Then, since \[
c_{x_{k+1}}(q)=\dots+(a_{i+1}+b_{i+1}+b_{i+2})q^{i+2}+(a_i+b_i+b_{i+1})q^{i+1}+(a_{i-1}+b_{i-1}+b_{i})q^{i}+\dots,
\]
we get $c_{i+1}\ge c_{i}\ge \dots$ and $c_{i+2}\ge c_{i+3}\ge \dots$ by $2$-peak of $c_{x_h}$ and $c_{x_{k-1}}$. In addition, \[
c_{i+1}-c_{i+3}=(a_i+b_i-a_{i+2}-b_{i+2})+(b_{i+1}-b_{i+3})\ge 0.
\]
Hence, $c_{x_{k+1}}$ is unimodal and has $(i,i+1)$ or $(i+1,i+2)$ $2$-peak which implies that $c_{x_{k+1}}$ satisfy induction hypothesis 1 and 2. 
\item[(ii)] The case when $c_{x_k}$ has $(i+1,i+2)$ $2$-peak
We have $a_{i+2}+b_{i+2}\ge a_{i}+b_{i}$. Let $c_{x_{k+1}}=c_{m''}u^{m''}+\dots+c_0$. Then, $c_i=a_{i-1}+b_{i-1}+b_i$ and we get $c_{i+1}\ge c_{i}\ge \dots$ and $c_{i+2}\ge c_{i+3}\ge \dots$ by $2$-peak of $c_{x_h}$ and $c_{x_{k-1}}$. In addition, \[
c_{i+2}-c_i=(a_{i+1}-a_{i-1})+(b_{i+1}+b_{i+2}-b_{i-1}-b_i)\ge 0.
\]
Hence, \( c_{x_{k+1}} \) is unimodal and has a \( 2 \)-peak at either \( (i+1,i+2) \) or \( (i+2,i+3) \), which implies that \( c_{x_k} \) and \( c_{x_{k+1}} \) satisfy the induction hypotheses 1 and 2.

One can observe that we can simply apply the same idea to other cases as well.
\end{proof}

A natural question is whether for a single lamination, the $c$-polynomial is almost interlacing. However, this is not true, as the following counterexample shows. 

\begin{example}
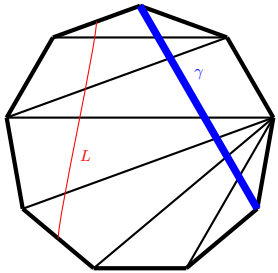
\begin{figure}[H]
    \centering
\begin{tikzpicture}[scale=0.6, transform shape]
    \def\n{9}
    \def\radius{3}

    \foreach \i in {1,...,\n} {
        \pgfmathsetmacro{\angle}{360/\n * \i + 50}
        \coordinate (V\i) at (\angle:\radius);
    }

    \foreach \i in {1,...,\n} {
        \pgfmathsetmacro{\nexti}{mod(\i, \n) + 1}
        \draw[ultra thick] (V\i) -- (V\nexti);
    }

    \draw[thick] (V2) -- (V9);
    \draw[thick] (V3) -- (V9);
    \draw[thick] (V3) -- (V8);
    \draw[thick] (V4) -- (V8);
    \draw[thick] (V5) -- (V8);
    \draw[thick] (V6) -- (V8);
    
    \draw[line width=1mm, blue] (V1) -- (V7) node[pos=0.4, above right] {$\gamma$};

    \coordinate (M1) at ($0.5*(V1) + 0.5*(V2)$);
    \coordinate (M2) at ($0.5*(V4) + 0.5*(V5)$);

    \draw[red, looseness=1] (M1) .. controls +(up:-0.5) and +(down:-0.5) .. (M2) node[pos=0.55, below right] {$L$};
\end{tikzpicture}

    \caption{Counterexample of \eqref{eq:ineqA} in single lamination}
    \label{fig:ineqasin}
\end{figure}

Figure~\ref{fig:ineqasin} shows that \eqref{eq:ineqA} need not hold for single laminations. In this example, the $c$-polynomial is $c_{x_\gamma}(q) = q^{2} + 6q + 7$.

\end{example}

We have demonstrated the unimodality of the coefficients of the rank polynomial for loop fence posets within the context of principal coefficients and plain arcs on unpunctured surfaces with a single lamination. To further investigate the properties of these polynomials, we propose the following conjecture regarding their log-concavity.

\begin{conjecture} \label{conj:log-concavity}
Let $(S,M)$ be an unpunctured disk and let $L$ be a single lamination. For any arc $\gamma$ on $(S,M)$, the coefficient sequence $(a_0, a_1, \ldots, a_d)$ of $c_{x_\gamma}(q) = \sum_{i=0}^d a_i q^i$ is log-concave.
\end{conjecture}

Through simulations, we have verified that the conjecture holds for numerous instances of single laminations on unpunctured surfaces, specifically for $n$-gons with $n$ up to 16. These results suggest that log-concavity may be an intrinsic property of single laminations.

\begin{example}

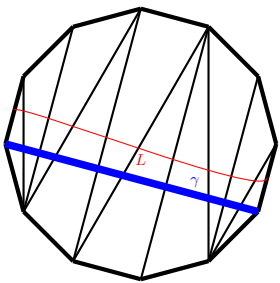
\begin{figure}[H]
\centering
\begin{tikzpicture}[scale=0.6, transform shape]
    \def\n{12}
    \def\radius{3}

    \foreach \i in {1,...,\n} {
        \pgfmathsetmacro{\angle}{360/\n * \i}
        \coordinate (V\i) at (\angle:\radius);
    }

    \foreach \i in {1,...,\n} {
        \pgfmathsetmacro{\nexti}{mod(\i, \n) + 1}
        \draw[ultra thick] (V\i) -- (V\nexti);
    }

    \draw[thick] (V1) -- (V10);
    \draw[thick] (V12) -- (V10);
    \draw[thick] (V2) -- (V10);
    \draw[thick] (V2) -- (V9);
    \draw[thick] (V2) -- (V8);
    \draw[thick] (V3) -- (V8);
    \draw[thick] (V3) -- (V7);
    \draw[thick] (V4) -- (V7);
    \draw[thick] (V5) -- (V7);

    \draw[line width=1mm, blue] (V6) -- (V11) node[pos=0.75, above] {$\gamma$};

    \coordinate (M1) at ($0.3*(V10) + 0.7*(V11)$);
    \coordinate (M2) at ($0.7*(V6) + 0.3*(V7)$);
    \coordinate (M5) at ($0.5*(V12) + 0.5*(V11)$);
    \coordinate (M6) at ($0.5*(V6) + 0.5*(V5)$);
    \coordinate (M7) at ($0.5*(V12) + 0.5*(V1)$);
    \coordinate (M8) at ($0.5*(V4) + 0.5*(V5)$);

    \draw[red, looseness=1] (M5) .. controls +(up:-0.5) and +(down:-0.2) .. (M6) node[pos=0.5, below] {$L$};
\end{tikzpicture}
    \caption{Example of triangulation of 16-gon with single lamination}
    \label{fig:single_lam}

\end{figure}

In figure \ref{fig:single_lam}, $c$-polynomial $c_{x_\gamma}$ of $\gamma$ is
\[c_{x_\gamma}(q)=2q^{8}+4q^7+6q^6+10q^5+11q^4+12q^3+9q^2+5q+2,\]
which shows that it is not only unimodal but also log-concave.

\end{example}

An intriguing aspect of this log-concavity is that the standard method using geometric lattices~\cite{AHK18}, a well-established tool for proving log-concavity, can not be applied here. The preceding example shows that the polynomial has a constant term equal to $2$ and a leading coefficient also equal to $2$. This indicates that it cannot arise from a geometric lattice.

We observe that when the lamination is not single—even in the case of elementary laminations—the resulting polynomial may fail to be log-concave (see the right side of Table~\ref{table:nonex}). However, for non-single laminations that do not permit repeated curves (that is, there are no two distinct lamination curves which are isotopic and share endpoints on the same boundary edges), the resulting polynomial often remains unimodal with internal zeros, even when the lamination is not elementary. If there are repeated curves in lamination, the unimodality fails (see the left side of Table~\ref{table:nonex}). Extending the analysis to such more general laminations is a natural next step, as the underlying combinatorial structure may reveal additional patterns in the coefficients of the multivariable expansions.

\renewcommand{\arraystretch}{1.2} 
\begin{table}[H]
\centering
\begin{tabular}{|c|c|}
\hline
\begin{tikzpicture}[scale=0.5, transform shape]
    \def\n{12}
    \def\radius{3}

    \foreach \i in {1,...,\n} {
        \pgfmathsetmacro{\angle}{360/\n * \i}
        \coordinate (V\i) at (\angle:\radius);
    }

    \foreach \i in {1,...,\n} {
        \pgfmathsetmacro{\nexti}{mod(\i, \n) + 1}
        \draw[ultra thick] (V\i) -- (V\nexti);
    }

    \draw[thick] (V1) -- (V10);
    \draw[thick] (V12) -- (V10);
    \draw[thick] (V2) -- (V10);
    \draw[thick] (V2) -- (V9);
    \draw[thick] (V2) -- (V8);
    \draw[thick] (V3) -- (V8);
    \draw[thick] (V3) -- (V7);
    \draw[thick] (V4) -- (V7);
    \draw[thick] (V5) -- (V7);

    \draw[line width=1mm, blue] (V6) -- (V11) node[pos=0.75, above] {$\gamma$};

    \coordinate (M1) at ($0.3*(V10) + 0.7*(V11)$);
    \coordinate (M2) at ($0.7*(V6) + 0.3*(V7)$);
    \coordinate (M3) at ($0.7*(V10) + 0.3*(V11)$);
    \coordinate (M4) at ($0.3*(V6) + 0.7*(V7)$);
    \coordinate (M5) at ($0.5*(V12) + 0.5*(V11)$);
    \coordinate (M6) at ($0.5*(V6) + 0.5*(V5)$);
    \coordinate (M7) at ($0.5*(V12) + 0.5*(V1)$);
    \coordinate (M8) at ($0.5*(V4) + 0.5*(V5)$);

    \draw[red, looseness=1] (M1) .. controls +(up:-0.5) and +(down:-0.5) .. (M2) node[pos=0.45, below] {$L$};
    \draw[red, looseness=1] (M3) .. controls +(up:-0.5) and +(down:-0.2) .. (M4);
    \draw[red, looseness=1] (M5) .. controls +(up:-0.5) and +(down:-0.2) .. (M6);
    \draw[red, looseness=1] (M7) .. controls +(up:0.5) and +(down:0.5) .. (M8);
\end{tikzpicture} & 
\begin{tikzpicture}[scale=0.5, transform shape]
    \def\n{12}
    \def\radius{3}

    \foreach \i in {1,...,\n} {
        \pgfmathsetmacro{\angle}{360/\n * \i}
        \coordinate (V\i) at (\angle:\radius);
    }

    \foreach \i in {1,...,\n} {
        \pgfmathsetmacro{\nexti}{mod(\i, \n) + 1}
        \draw[ultra thick] (V\i) -- (V\nexti);
    }

    \draw[thick] (V1) -- (V10);
    \draw[thick] (V12) -- (V10);
    \draw[thick] (V2) -- (V10);
    \draw[thick] (V2) -- (V9);
    \draw[thick] (V2) -- (V8);
    \draw[thick] (V3) -- (V8);
    \draw[thick] (V3) -- (V7);
    \draw[thick] (V4) -- (V7);
    \draw[thick] (V5) -- (V7);

    \draw[line width=1mm, blue] (V6) -- (V11) node[pos=0.75, above] {$\gamma$};

    \coordinate (M1) at ($0.3*(V10) + 0.7*(V11)$);
    \coordinate (M2) at ($0.7*(V6) + 0.3*(V7)$);
    \coordinate (M5) at ($0.5*(V12) + 0.5*(V11)$);
    \coordinate (M6) at ($0.5*(V6) + 0.5*(V5)$);
    \coordinate (M7) at ($0.5*(V12) + 0.5*(V1)$);
    \coordinate (M8) at ($0.5*(V4) + 0.5*(V5)$);

    \draw[red, looseness=1] (M1) .. controls +(up:0.5) and +(down:0.5) .. (M2) node[pos=0.5, below] {$L$};
    \draw[red, looseness=1] (M5) .. controls +(up:-0.5) and +(down:-0.2) .. (M6);
    \draw[red, looseness=1] (M7) .. controls +(up:0.5) and +(down:0.5) .. (M8);
\end{tikzpicture} \\
\hline
$\begin{array}{l}
c_{x_\gamma}(q)=2q^{10}+4q^9+6q^8+4q^7+8q^6\\
\quad+9q^5+12q^4+6q^3+6q^2+2q+2
\end{array}$ & 
$\begin{array}{l}
c_{x_\gamma}(q)=2q^{8}+4q^7+6q^6+10q^5\\
\quad+11q^4+12q^3+9q^2+5q+2
\end{array}$ \\
\hline
\end{tabular}
\caption{Non-examples of unimodality/log-concavity}
\label{table:nonex}
\end{table}

\subsection*{Acknowledgements}
This project benefited from conversations with Kyungyong Lee and Ralf Schiffler. We would also like to thank Esther Banaian for useful discussions. W.K. was partially supported by the Simons Foundation, grant SFI-MPS-T-Institutes-00007697, and the Bulgarian Ministry of Education and Science grant DO1-239/10.12.2024. K.L. was supported by the National Research Foundation of Korea(NRF) grant funded by the Korea government(MSIT) (No. RS-2025-02262988). We are grateful to Adam Til\`epe for pointing out the gap in section 5.

\printbibliography

\vspace{1cm}

International Center for Mathematical Sciences, Institute of Mathematics and Informatics, Bulgarian Academy of Sciences, Acad. G. Bonchev Str., Bl. 8, Sofia 1113, Bulgaria
\\
\textit{Email address: }\href{mailto:wonk@math.bas.bg}{\texttt{wonk@math.bas.bg}}\newline

Department of Mathematical Sciences, Ulsan National Institute of Science and Technology, Ulsan 44919, Republic of Korea\\
\textit{Email address: }\href{mailto:kjlee@unist.ac.kr}{\texttt{kjlee@unist.ac.kr}}\newline

Department of Mathematics, Yonsei University, 50 Yonsei-Ro, Seodaemun-Gu, Seoul 03722, Korea\\
\textit{Email address: }\href{mailto:eunsung@yonsei.ac.kr}{\texttt{eunsung@yonsei.ac.kr}}\newline

\end{document}